\documentclass[12pt]{amsart}

\usepackage{amsthm,amssymb,amsmath,amsfonts, tikz, graphicx, a4wide,subfigure,bm}
\usepackage[english]{babel}
\usepackage{subfigure}
\usepackage{mathtools}
\usepackage{url}
\usepackage{enumerate}
\usepackage[hidelinks]{hyperref}

\DeclarePairedDelimiter\floor{\lfloor}{\rfloor}

\usepackage{xargs}   
\usepackage{xcolor}
\usepackage[colorinlistoftodos,prependcaption]{todonotes}
\newcommandx{\unsure}[2][1=]{\todo[linecolor=red,backgroundcolor=red!25,bordercolor=red,#1]{#2}}
\newcommandx{\change}[2][1=]{\todo[linecolor=blue,backgroundcolor=blue!25,bordercolor=blue,#1]{#2}}
\newcommandx{\info}[2][1=]{\todo[linecolor=OliveGreen,backgroundcolor=OliveGreen!25,bordercolor=OliveGreen,#1]{#2}}
\newcommandx{\improvement}[2][1=]{\todo[linecolor=Plum,backgroundcolor=red!25,bordercolor=red,#1]{#2}}
\newcommandx{\thiswillnotshow}[2][1=]{\todo[disable,#1]{#2}}

\begin{document}

\newtheorem{prop}{Proposition}[section]
\newtheorem{theorem}{Theorem}[section]
\newtheorem{lemma}{Lemma}[section]
\newtheorem{cor}{Corollary}[section]
\newtheorem{remark}{Remark}[section]
\theoremstyle{definition}
\newtheorem{defn}{Definition}[section]
\newtheorem{ex}{Example}[section]

\numberwithin{equation}{section}

\title{Critical intermittency in random interval maps}
\author[Homburg, Kalle, Ruziboev, Verbitskiy, Zeegers]{Ale Jan Homburg, Charlene Kalle, Marks Ruziboev, Evgeny Verbitskiy and Benthen Zeegers}

\address{A.J. Homburg\\ KdV Institute for Mathematics, University of Amsterdam, Science park 107, 1098 XG Amsterdam, Netherlands\newline Department of Mathematics, VU University Amsterdam, De Boelelaan 1081, 1081 HV Amsterdam, Netherlands}
\email{a.j.homburg@uva.nl}

\address{C.C.C.J. Kalle\\ Mathematical Institute, University of Leiden, PO Box 9512, 2300 RA Leiden, The Netherlands}
\email{kallecccj@math.leidenuniv.nl}

\address{M. Ruziboev\\ Mathematical Institute, University of Leiden, PO Box 9512, 2300 RA Leiden, The Netherlands\newline Faculty of Mathematics, University of Vienna, Oskar-Morgnstern Platz 1, Austria}
\email{marks.ruziboev@univie.ac.at}

\address{E.A. Verbitskiy\\ Mathematical Institute, University of Leiden, PO Box 9512, 2300 RA Leiden, The Netherlands\newline   Bernoulli Institute, University of Groningen, Nijenborgh 9, 9747 AG Groningen, The Netherlands}
\email{evgeny@math.leidenuniv.nl}

\address{B.P. Zeegers\\ Mathematical Institute, University of Leiden, PO Box 9512, 2300 RA Leiden, The Netherlands}
\email{b.p.zeegers@math.leidenuniv.nl}

\date{Version of \today}

\begin{abstract} Critical intermittency stands for a type of intermittent dynamics in iterated function systems, caused by an interplay of a superstable fixed point and a repelling fixed point. 
	We consider critical intermittency for iterated function systems of interval maps
	and demonstrate the existence of a phase transition when varying probabilities, where the absolutely continuous stationary measure changes between finite and infinite. We discuss further properties of this stationary measure and show that its density is not in $L^q$ for any $q>1$. 
	This provides a theory of critical intermittency alongside the theory for the well studied Manneville-Pomeau maps, where the intermittency is caused by a neutral fixed point.
\end{abstract}

\maketitle

\section{Introduction}

Intermittency refers to the behaviour of a dynamical system that alternates between long periods of exhibiting one out of several types of dynamical characteristics. In their seminal paper \cite{ManPum} Manneville and Pomeau investigated intermittency in the context of transitions to turbulence in convective fluids, see also \cite{ManPum2,BPV}, and distinguished several different types of intermittency. An illustrative example of a one-dimensional map with intermittent behaviour is the Manneville-Pomeau map
\[ T: [0,1]\to [0,1], \, x \mapsto x+x^{1+\alpha} \pmod 1\]
for some $\alpha>0$. The source of intermittency for this map is the presence of a neutral fixed point at the origin, which causes orbits to spend long periods of time close to zero, while behaving chaotically once they escape. 
 
\medskip 
The dynamics of the Manneville-Pomeau map and similar maps with a single neutral fixed point have been extensively studied over the past decades. It is known for example that such maps admit an absolutely continuous invariant measure (acim) and that their statistical properties are determined by the characteristics of the fixed point. See~\cite{LSV,PW99,Y99,G004,G007,BT16,BS16,FFTV} for results on Manneville-Pomeau type maps, and \cite{Tha80,CF90,Hu01,PY01,Zwe03} for other related results on one-dimensional systems with neutral fixed points.

\medskip
Intermittency caused by neutral fixed points was also studied in random dynamical systems, see e.g.~\cite{BBD14,BB16,KKV17,BBR19,BRS20,KMTV}. These results show that a random dynamical system, built as a mixture of `good' maps with finite acim's and `bad' maps with slower mixing rates or without finite acim's, inherit ergodic properties typical for the `good' maps: e.g., the random systems still admit a unique finite acim.
On the other hand, it is clear that in the random mixture of good and bad maps, the presence of bad maps should be visible in the properties of the acim. In \cite{KMTV} it was shown that in the random system built using the `good' Gauss and `bad' R\'enyi continued fractions maps, the density of the acim
is provably less smooth than the invariant density of the Gauss map. This loss of smoothness is an interesting new phenomenon, which deserves further study.

\medskip
The topic of the present paper is another type of intermittency observed in random dynamical systems: 
 the so-called {\em critical intermittency} introduced recently in \cite{AbbGhaHom,HP}. To illustrate the concept, consider the Markov process generated by random applications of one of the two logistic maps $T_2(x) = 2x(1-x)$ and $T_4 (x) = 4x(1-x)$: for each $n$, independently 
\[  x_{n+1}=\begin{cases} T_2(x_n),&\ \text{ with prob. }p_2,\\
 T_4(x_n),&\ \text{ with prob. }p_4=1-p_2.
 \end{cases}\]
The dynamics of these two maps individually is quite different: $T_4$ exhibits chaotic behaviour and admits an ergodic absolutely continuous invariant probability measure, while $T_2$ has $\frac12$ as a superattracting fixed point with $(0,1)$ as its basin of attraction. 
Under random compositions of $T_2$ and $T_4$ the typical behaviour is the following: orbits are quickly attracted to $\frac{1}{2}$ by applications of $T_2$ and are then repelled first close to 1 and then close to $0$ by one application of $T_4$ followed by an application of either $T_2$ or $T_4$. Since 0 is a repelling fixed point for both maps, orbits then leave a neighbourhood of 0 after a number of time steps, see Figure \ref{fig:interm}. This pattern occurs infinitely often in typical random orbits and is the result of the interplay between the exponential divergence from 0 under $T_2$ and $T_4$ and the superexponential convergence to $\frac12$ under $T_2$. Figure~\ref{fig:inter}(c) shows an orbit under random compositions of $T_2$ and $T_4$ as well as an orbit of a point under a Manneville-Pomeau map in (a) and a random orbit under compositions of the Gauss and R\'enyi maps in (b).

\begin{figure}[h]
\begin{tikzpicture}[scale =4]
\draw(-.01,0)--(1.01,0)(0,-.01)--(0,1.01);
\draw[dotted](.5,0)--(.5,1)(0,1)--(.5,1)(.5,.5)--(0,.5);
\draw[dotted](0,0)--(1,1);
\draw[line width=.4mm, green!50!blue!70!black, smooth, samples =20, domain=0:1] plot(\x, { 4* \x * (1-\x)});
\draw[line width=.4mm, green!50!blue!70!black, smooth, samples =20, domain=0:1] plot(\x, { 2* \x * (1-\x)});
\draw[red, dashed, line width=.25mm](.25,0)--(.255,.375)--(.375,.375)--(.37995,.4688)--(.4688,.4688)--(.4688,.9961)--(.9961,.9961)--(.9961,.0155)--(.0155,.0155)--(.0155,.0262)--(.0262,.0262)--(0.0262,.051)--(0.051,.051)--(0.051,.1975)--(0.1975,.1975);

\node[below] at (.25,0){\small $x$};
\node[below] at (-.01,0){\tiny 0};
\node[below] at (1,0){\tiny 1};
\node[below] at (.5,0){\tiny $\frac12$};
\node[left] at (0,1){\tiny 1};
\node[left] at (0,.5){\tiny $\frac12$};
\end{tikzpicture}
\caption{Critical intermittency in the random system of logistic maps $T_2$, $T_4$. The dashed line indicates part of a random orbit of $x$}
\label{fig:interm}
\end{figure}
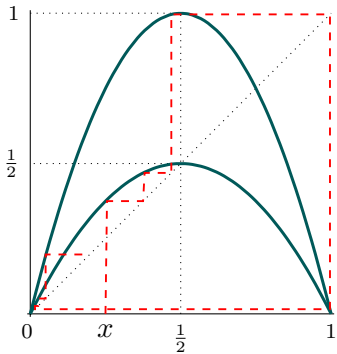

\begin{figure}[ht]
\includegraphics[width=0.9\textwidth]{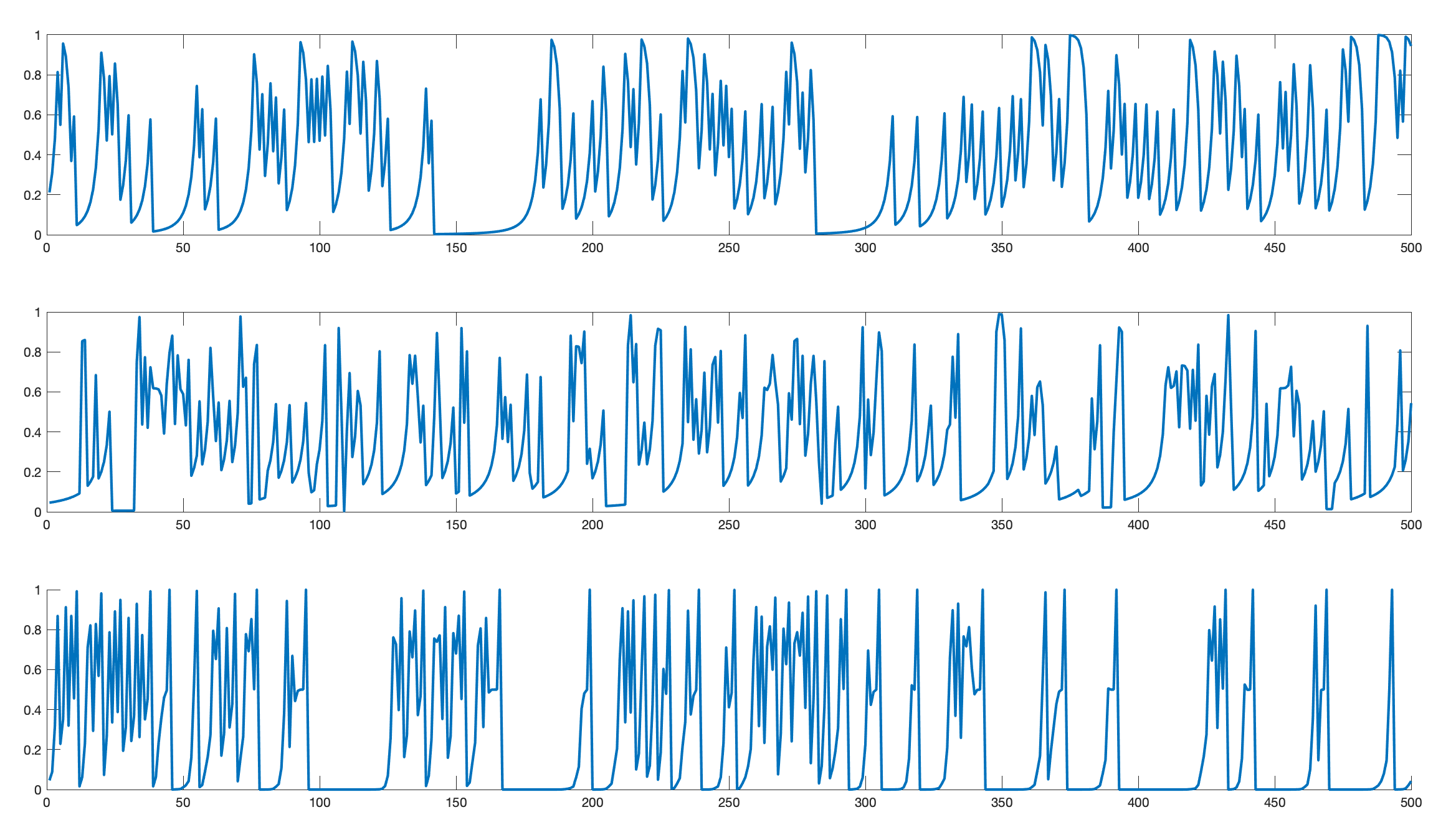}
\caption{Intermittent behaviour of orbits of (a) a single Manneville-Pomeau map with $\alpha=1.5$, (b) a random mixture of the Gauss and R\'enyi continued fractions maps where the Gauss map is chosen with probability $p=0.1$ and (c) a random mixture of the logistic maps $T_2$ and $T_4$ where the map $T_4$ is chosen with probability $p=0.6$.}
\label{fig:inter}
\end{figure} 

\medskip
The dynamical behaviour of random compositions of the two logistic maps $T_2$ and $T_4$ was studied in \cite{AD00,AthSch,AbbGhaHom,Car02,HP} among others. In \cite{AbbGhaHom,HP} the authors investigated the existence and finiteness of absolutely continuous invariant measures for this random system and for iterated function systems consisting of rational maps on the Riemann sphere. One particular result from \cite{AbbGhaHom} states that the random dynamical system generated by i.i.d.~compositions of $T_2$ and $T_4$ chosen with probabilities $p_2$ and $p_4 = 1-p_2$ admits an absolutely continuous invariant measure that is $\sigma$-finite on the interval $[0,1]$ and that is infinite in case $p_2>\frac12$. An interesting question that was left open in \cite{AbbGhaHom} is whether for $p_2 \leq \frac 12$ this measure is infinite or finite.

\medskip
In this article we answer this question. We consider a large family of random interval maps with critical intermittency that includes the random combination of $T_2$ and $T_4$. The systems we consider consist of i.i.d.~compositions of a finite number of maps of two types: bad maps which share a superattracting fixed point and good maps that map the superattacting fixed point onto a common repelling fixed point. To be precise, the families of maps we consider are defined as follows.

\medskip
Throughout the text we fix a point $c \in (0,1)$ that will represent the single critical point of our maps, both good and bad.

\medskip
A map $T_g:[0,1] \to [0,1]$ is in the class of {\em good maps}, denoted by $\mathfrak{G}$, if
\begin{enumerate}[(G1)]
\item $T_g|_{(0,c)}$ and  $T_g|_{(c,1)}$ are $C^3$ diffeomorphisms onto $(0,1)$ and $T_g(\{0,c, 1\}) \subseteq \{0,1\}$;
\item $T_g$ has non-positive Schwarzian derivative on $[0,c)$ and $(c,1]$;
\item to $T_g$ we can associate three constants $r_g \geq 1$, $0< K_g<1$ and $M_g > r_g$ such that
\begin{align}\label{eqn2.1}
K_g |x-c|^{r_g-1} \leq |DT_g(x)| \leq M_g |x-c|^{r_g-1};
\end{align}
\item we have $|DT_g(0)|,|DT_g(1)| > 1$.
\end{enumerate}
These conditions imply in particular that at least one of the maps $T_g|_{[0,c]}$ or $T_g|_{[c,1]}$ is continuous, and that both branches of $T_g$ are strictly monotone. Note also that the conditions $K_g <1$ and $M_g > r_g$ are superfluous, since we can always choose a smaller constant $K$ and larger constant $M$ to satisfy \eqref{eqn2.1}, but we need these specific bounds in our estimates later. The critical point $c$ is mapped to either 0 or 1 under each of the good maps and both 0 and 1 are (eventually) fixed points or periodic points (with period 2) by (G1) that are repelling by (G4). Examples include the doubling map and any surjective unimodal map, see Figures~\ref{T:examples}(a) and (b).

\medskip
The choice of conditions (G1)-(G4) is based on two factors: firstly, these conditions incorporate the most important properties of the `good' logistic map $T_4(x)=4x(1-x)$, which is the primary motivating example for this work,  and secondly, the techniques used in this paper are motivated by the work of Nowicki and Van Strien \cite{NowvSt} where the following result has been proven. Throughout the text we let $\lambda$ denote the one-dimensional Lebesgue measure.

\begin{theorem}\label{thrm1.1}
Suppose that $T:[0,1] \to [0,1]$ is unimodal, $C^3$, has negative Schwarzian derivative and that the critical point of $T$ is of order $r \ge 1$. Moreover assume that the growth rate of $\left|D T^{n}\left(c_{1}\right)\right|$, $c_1=T(c)$,  is so fast that
\begin{equation}\label{eq:growth}
\sum_{n=0}^{\infty}\left|D f^{n}\left(c_{1}\right)\right|^{-1 / r}<\infty.
\end{equation}
Then $T$ has a unique absolutely continuous invariant probability measure $\mu$ which is ergodic and of positive entropy. Furthermore, there exists a positive constant $K$ such that
\begin{equation}\label{eq:NvS}
\mu(A) \le K\lambda(A)^{1 / r},
\end{equation}
for any measurable set $A \subset(0,1)$. Finally, the density $\rho=\frac {d\mu}{d\lambda}$ of the measure $\mu$ with respect to $\lambda$ is an $L^{\mathrm{\tau}-}$-function where $\tau=r /(r-1)$ and $L^{\tau-}=\bigcap_{1 \leqq t<\tau} L^{t}$ and $L^{t}=\big\{\rho \in L^{1}: \int_0^1|\rho|^{t} d \lambda<\infty \big\}$.
\end{theorem}

Formally this result is not immediately applicable to the good maps we introduced. The difference, however, is not principal and the conclusion remains exactly the same, the main reason being that the conditions (G1) and (G4)  imply the growth rate \eqref{eq:growth}, and hence any good map admits a unique probability acim.

\begin{figure}[h]
\subfigure[\tiny $x \mapsto 2x \pmod 1$]{
\begin{tikzpicture}[scale =2]
\draw(-.01,0)--(1.01,0)(0,-.01)--(0,1.01);
\draw[dotted](.5,0)--(.5,1)(0,1)--(1,1);
\draw[line width=.4mm, green!50!blue!70!black] (0,0)--(.5,1)(.5,0)--(1,1);
\node[below] at (1,0){\tiny 1};
\node[below] at (.5,0){\tiny $\frac12$};
\node[below] at (-.01,0){\tiny 0};
\node[left] at (0,1){\tiny 1};
\end{tikzpicture}
}
\hspace{.5cm}
\subfigure[\tiny $x \mapsto 4x(1-x)$]{
\begin{tikzpicture}[scale =2]
\draw(-.01,0)--(1.01,0)(0,-.01)--(0,1.01);
\draw[dotted](.5,0)--(.5,1)(0,1)--(.5,1);
\draw[line width=.4mm, green!50!blue!70!black, smooth, samples =20, domain=0:1] plot(\x, { 4* \x * (1-\x)});
\node[below] at (1,0){\tiny 1};
\node[below] at (.5,0){\tiny $\frac12$};
\node[below] at (-.01,0){\tiny 0};
\node[left] at (0,1){\tiny 1};
\end{tikzpicture}
}
\hspace{.5cm}
\subfigure[\tiny $x \mapsto \frac12 - 4\big(x-\frac12\big)^3 $]{
\begin{tikzpicture}[scale =2]
\draw(-.01,0)--(1.01,0)(0,-.01)--(0,1.01);
\draw[dotted](.5,0)--(.5,.5)(.5,.5)--(0,.5);
\draw[line width=.4mm, green!50!blue!70!black, smooth, samples =20, domain=0:1] plot(\x, { 0.5 - 4*( \x-0.5)^3});
\node[below] at (1,0){\tiny 1};
\node[below] at (.5,0){\tiny $\frac12$};
\node[below] at (-.01,0){\tiny 0};
\node[left] at (0,1){\tiny 1};
\node[left] at (0,.5){\tiny $\frac12$};
\end{tikzpicture}
}
\hspace{.5cm}
\subfigure[\tiny $x \mapsto 2x(1-x)$]{
\begin{tikzpicture}[scale =2]
\draw(-.01,0)--(1.01,0)(0,-.01)--(0,1.01);
\draw[dotted](.5,0)--(.5,.5)(.5,.5)--(0,.5);
\draw[line width=.4mm, green!50!blue!70!black, smooth, samples =20, domain=0:1] plot(\x, { 2* \x * (1-\x)});
\node[below] at (1,0){\tiny 1};
\node[below] at (.5,0){\tiny $\frac12$};
\node[below] at (-.01,0){\tiny 0};
\node[left] at (0,1){\tiny 1};
\node[left] at (0,.5){\tiny $\frac12$};
\end{tikzpicture}
}
\caption{Four maps with critical point $c = \frac12$. (a) and (b) show two good maps, while in (c) and (d) we see the graphs of two bad maps.}
\label{T:examples}
\end{figure}
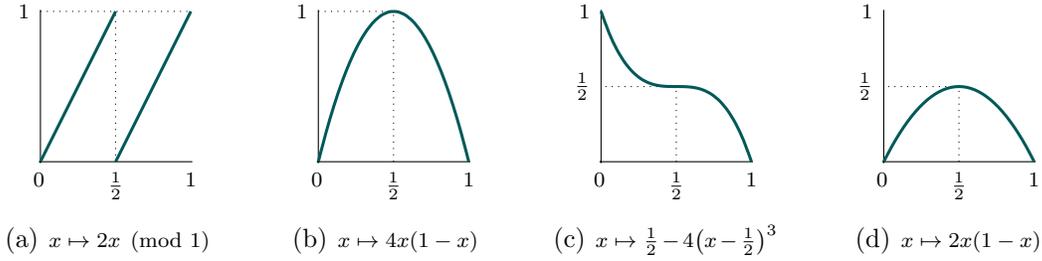

\medskip
A map $T_b:[0,1]\to [0,1]$ is in the class of {\em bad maps}, denoted by $\mathfrak{B}$, if
\begin{enumerate}[(B1)]
\item $T_b|_{(0,c)}$ and $T_b|_{(c,1)}$ are $C^3$ diffeomorphisms onto $(0,c)$ or $(c,1)$, $T_b(\{0,1\}) \subseteq \{0,1\}$ and $T_b(c) = c$;
\item $T_b$ has non-positive Schwarzian derivative on $[0,c)$ and $(c,1]$;
\item to $T_b$ we can associate three constants $\ell_b > 1$, $0< K_b< 1$ and $M_b > \ell_b$ such that
\begin{align}\label{eqn2.2}
K_b |x-c|^{\ell_b-1} \leq |DT_b(x)| \leq M_b |x-c|^{\ell_b-1};
\end{align}
\item we have $|DT_b(0)|,|DT_b(1)| > 1$.
\end{enumerate}
In particular (B1) implies that $T_b$ is continuous, and that $T_b$ strictly monotone on the intervals $[0,c]$ and $[c,1]$. In contrast to (G3), note that in (B3) we have assumed that $\ell_b$ is not equal to one. This means that $DT_b(c) = 0$, so $c$ is a superattracting fixed point for each bad map. An immediate consequence of the presence of a globally attracting fixed point at $c$ is that the only finite invariant measures are linear combinations of Dirac measures at $0,c$, and $1$. For examples, see Figures~\ref{T:examples}(c) and (d).

\medskip
The random systems we consider in this article are the following. Let $T_1,\ldots,T_N \in \mathfrak G \cup \mathfrak B$ be a finite collection of good and bad maps. Write $\Sigma_G = \{1 \leq j \leq N\, :\,  T_j \in \mathfrak G\}$ and $\Sigma_B = \{1 \leq j \leq N\, :\,  T_j \in \mathfrak B\}$ for the index sets of the good and bad maps respectively and assume that $\Sigma_G,\Sigma_B \neq \emptyset$. Write $\Sigma = \{ 1, \ldots, N \} = \Sigma_G \cup \Sigma_B$. The skew product transformation  or {\em random map} $F$ is defined by
\begin{equation}\label{q:skewproduct}
 F:\Sigma^{\mathbb N} \times [0,1] \to \Sigma^{\mathbb N} \times [0,1], \, (\omega,x) \mapsto (\sigma \omega, T_{\omega_1}(x)),
 \end{equation}
where $\sigma$ denotes the left shift on sequences in $\Sigma^{\mathbb N}$. Let $\mathbf p = (p_j)_{j \in \Sigma}$ be a probability vector representing the probabilities with which we choose the maps $T_j$, $j \in \Sigma$. We will consider measures of the form $\mathbb P \times \mu_{\mathbf p}$, where $\mathbb P$ is the $\mathbf p$-Bernoulli measure on $\Sigma^\mathbb N$ and $\mu_{\mathbf p}$ is a Borel measure on $[0,1]$ absolutely continuous with respect to $\lambda$ and  satisfying
\begin{align}\label{eqn9}
\sum_{j \in \Sigma} p_j \mu_{\mathbf p}(T_j^{-1}A) = \mu_{\mathbf p}(A), \qquad \text{for all Borel sets $A \subseteq [0,1]$}.
\end{align}
In this case $\mathbb P \times \mu_{\mathbf p}$ is an invariant measure for $F$ and we say that $\mu_{\mathbf p}$ is a {\em stationary} measure for $F$. We also say that a stationary measure $\mu_{\mathbf p}$ is ergodic for $F$ if $\mathbb P \times \mu_{\mathbf p}$ is ergodic for $F$. Our main results are the following.

\begin{theorem}\label{MAIN} Let $\{T_j: j \in \Sigma\}$ be as above and $\mathbf p = (p_j)_{j \in \Sigma}$ a positive probability vector.
 \begin{enumerate}
\item There exists a unique (up to scalar multiplication) stationary $\sigma$-finite measure $\mu_{\mathbf p}$ for $F$ that is absolutely continuous with respect to the one-dimensional Lebesgue measure $\lambda$. Moreover, this measure is ergodic.
\item The density $\frac{d\mu_{\mathbf p}}{d\lambda}$ is bounded away from zero, is locally Lipschitz on $(0,c)$ and $(c,1)$ and is not in $L^q$ for any $q > 1$.
\end{enumerate}
\end{theorem}

\noindent We call the measure $\mu_{\mathbf p}$ from Theorem~\ref{MAIN} an {\em acs} measure.

\begin{theorem}\label{MAIN2} Let $\{T_j: j \in \Sigma\}$ be as above and $\mathbf p = (p_j)_{j \in \Sigma}$ a positive probability vector. Let $\mu_{\mathbf p}$ be the unique acs measure from Theorem~\ref{MAIN}. Set $\theta = \sum_{b \in \Sigma_B} p_b \ell_b$. Then $\mu_{\mathbf p}$ is finite if and only if $\theta<1$. In this case, there exists a constant $C > 0$ such that
\begin{align}\label{eq:2.10}
\mu_{\mathbf p}(A) \leq C \cdot \sum_{k=0}^{\infty} \theta^k \lambda(A)^{\ell_{\max}^{-k} r_{\max}^{-1}}
\end{align}
for any Borel set $A \subseteq [0,1]$, where $r_{\max} = \max\{r_g: g \in \Sigma_G\}$ and $\ell_{\max} = \max\{\ell_b: b \in \Sigma_B\}$.
\end{theorem}

As we shall see in \eqref{eqn3.42} the bound in \eqref{eq:2.10} can be improved by not bounding mixtures $\ell_{\mathbf{b}} r_g = \prod_{i=1}^k \ell_{b_i} r_g$ by their maximal value $\ell_{\max}^k r_{\max}$, but this improvement does not change the qualitative behaviour of the bound.

\medskip
It will become clear that the density $\frac{d\mu_{\mathbf p}}{d\lambda}$ in Theorem \ref{MAIN} blows up to infinity at the points zero and one and also (at least on one side) at $c$. Theorem \ref{MAIN2} says that $\frac{d\mu_{\mathbf p}}{d\lambda}$ is integrable if and only if $\theta$ is small enough, namely $\theta < 1$. This intuitively makes sense since for a smaller value of $\theta$ the attraction of orbits to $c$ is weaker on average and consequently orbits typically spend less time near zero and one once a good map is applied.

\medskip
The inequality \eqref{eq:2.10} is the counterpart of the Nowicki-Van Strien inequality \eqref{eq:NvS}, and naturally gives a substantially worse bound due to the presence of bad maps. It is not immediately clear how much worse \eqref{eq:2.10} is in comparison to \eqref{eq:NvS}. However, the following holds.

\begin{cor} \label{cor2}
Let $ \{T_j: j \in \Sigma\}$ be as above and $\mathbf p = (p_j)_{j \in \Sigma}$ a positive probability vector. Suppose $\theta = \sum_{b \in \Sigma_B} p_b \ell_b < 1$. Then there exist $K > 0$ and $\varkappa > 0$ such that for any Borel set $A \subseteq [0,1]$ with $\lambda(A) \in (0,1)$ one has
\[ \mu_{\mathbf p}(A)\le  K \frac {1}{ \log^\varkappa (1/\lambda(A))}.\]
\end{cor}

\medskip
Moreover, the acs measure $\mu_{\mathbf p}$ from Theorem~\ref{MAIN} depends continuously on the probability vector $\mathbf p \in \mathbb R^N$.

\begin{cor} \label{cor}
Let $ \{T_j: j \in \Sigma\}$ be as above. For each $n \ge 0$, let $\mathbf p_n = (p_{n,j})_{j \in \Sigma}$ be a positive probability vector such that $\sup_{n}\sum_{b \in \Sigma_B} p_{n, b}  \ell_b<1$ and assume that $\lim_{n\to \infty}\mathbf p_n=\mathbf p$ in $\mathbb R_+^N$. Then the sequence $\mu_{\mathbf p_n}$ converges weakly to $\mu_{\mathbf p}$.
\end{cor} 

\medskip
In (B3) we have assumed that for any bad map $T_b$ the corresponding value $\ell_b$ is not equal to one. Note that a bad map $T_b$ for which we allow $\ell_b = 1$ satisfies $|DT_b(c)| > 0$, so in this case $c$ is an attracting fixed point for $T_b$ but not superattracting. It should not come as a surprise that results similar to Theorem~\ref{MAIN} and Theorem~\ref{MAIN2} also hold in case some or all of the bad maps $T_b$ have $\ell_b=1$. The proofs presented for these theorems, however, do not immediately carry over. In the last section we explain how the results are affected in case some or all maps $T_b$ satisfy $\ell_b=1$ and what the necessary changes in the proofs are. 

\medskip
The proofs use a mixture of techniques. For the existence result from Theorem~\ref{MAIN} we use an inducing scheme. This approach is inspired by \cite{AbbGhaHom}, but the choice of the inducing domain needed some care. With the help of Kac's Lemma we then obtain that the acs measure is infinite in case $\theta \ge 1$. To prove that this measure is finite for $\theta <1$ we use an approach similar to the one employed in \cite{NowvSt}. The main difficulty here is that it may take an arbitrarily long time before the superattracting fixed point is mapped onto the repelling orbit by one of the good maps, which decreases the regularity of the density of the acs measure.
 
\medskip
The paper is organised as follows. In Section \ref{sec:mr} we list some preliminaries and first consequences of the conditions (G1)--(G4) and (B1)--(B4). Section \ref{sec:PT2.1} is devoted to the proof of Theorem~\ref{MAIN} and in Section~\ref{s:in-finite} we prove Theorem~\ref{MAIN2}. In Section~\ref{sec:cr} we prove Corollaries \ref{cor2} and \ref{cor} and explain what the analogues of Theorem~\ref{MAIN} and \ref{MAIN2} are in case $\ell_b=1$ for one or more $b \in \Sigma$ and how the proofs of Theorem~\ref{MAIN} and \ref{MAIN2} need to be modified to get these results. We end with some final remarks.

\section*{Acknowledgments}
We would like to thank Sebastian van Strien for useful suggestions.

\section{Preliminaries}\label{sec:mr}

We start by introducing some notation and collecting some general preliminaries.

\subsection{Words, sequences and invariant measures}
For any finite subset $\Sigma \subseteq \mathbb N$ and any $n \ge 1$ we use $\mathbf u \in \Sigma^n$ to denote a {\em word} $\mathbf u = u_1 \cdots u_n$. $\Sigma^0$ contains only the empty word, which we denote by $\epsilon$. On the space of infinite sequences $\Omega = \Sigma^\mathbb N$ we use
\[ [\mathbf u] = [u_1 \cdots u_n] = \{\omega \in \Omega: \omega_1 = u_1, \ldots, \omega_n = u_n\}\]
to denote the {\em cylinder set} corresponding to $\mathbf u$. The notation $|\mathbf u|$ indicates the length of $\mathbf u$, so $|\mathbf u|=n$ for $\mathbf u \in \Sigma^n$. For two words $\mathbf u \in \Sigma^n$ and $\mathbf v \in \Sigma^m$ the concatenation of $\mathbf u$ and $\mathbf v$ is denoted by $\mathbf{uv} \in \Sigma^{n+m}$. For a probability vector $p = (p_j)_{j \in \Sigma}$ and $\mathbf u \in \Sigma^n$ we write $p_{\mathbf u} = \prod_{i=1}^n p_{u_i} $ with $p_{\mathbf u}=0$ if $n=0$.  We use $\sigma$ to denote the {\em left shift} on $\Omega$: for $\omega\in\Omega$ and all $n\in \mathbb N$, $(\sigma \omega)_n = \omega_{n+1}$. 

\medskip
Given a finite family of Borel measurable maps $\{ T_j: [0,1] \to [0,1]\}_{j \in \Sigma}$, the skew product or the random map $F$ is defined by\[ F: \Omega \times [0,1] \to \Omega \times [0,1], \, (\omega,x) \mapsto \big( \sigma \omega, T_{\omega_1}(x)\big).\]
We use the following notation for the iterates of the maps $T_j$. For each $\omega \in \Omega$ and each $n \in \mathbb{N}_0$ define
\begin{align}
T_{\omega_1 \cdots \omega_n }(x) = T_\omega^n(x) =  \begin{cases}  x, & \text{if}\quad n=0,\\
T_{\omega_n}\circ  T_{\omega_{n-1}}\circ \cdots \circ T_{\omega_1}(x), & \text{for } n\ge 1.
\end{cases}
\end{align}
With this notation, we can write the iterates of the random system $F$ as
\begin{align}
F^n(\omega,x) = (\sigma^n \omega, T_{\omega}^n(x)).
\end{align}

The following lemma on invariant measures for $F$ holds.
\begin{lemma}[\cite{Mor85}, see also Lemma 3.2 of \cite{Fro99}]\label{l:productmeasure}
If all maps $T_j$ are non-singular with respect to $\lambda$ (that is, $\lambda(A) =0$ if and only if $\lambda(T_j^{-1}A) = 0$ for all $A \subseteq [0,1]$ measurable) and $\mathbb P$ is the $\mathbf p$-Bernoulli measure on $\Omega$ for some positive probability vector $\mathbf p$, then the $\mathbb{P} \times \lambda$-absolutely continuous $F$-invariant measures are precisely the measures of the form $\mathbb{P} \times \mu$ where $\mu$ is absolutely continuous w.r.t.~$\lambda$ and satisfies
\begin{align}\label{eqn9}
\sum_{j \in \Sigma} p_j \mu (T_j^{-1}A) = \mu (A) \qquad \text{for all Borel sets $A$}.
\end{align}
\end{lemma}

Now let $(X, \mathcal F, m)$ be a measure space and $T: X \to X$ measurable and non-singular with respect to $m$. For a set $Y \in \mathcal F$ such that  $0 < m(Y) < \infty$ and $m \big( X \setminus \bigcup_{n \ge 1} T^{-n} Y \big)=0$, the {\em first return time map} $\varphi_Y: Y \to \mathbb N \cup \{ \infty\}$ given by
\begin{equation}\label{q:frtm}
\varphi_Y(y) = \inf \{ n \ge 1 \, : \,  T^n(y) \in Y \}
\end{equation}
is finite $m$-a.e.~on $Y$,  and moreover $m$-a.e.~$y\in Y$ returns to $Y$ infinitely often. If we remove from $Y$ the $m$-null set of points that return to $Y$ only finitely many times, and for convenience call this set $Y$ again, then we can define the {\em induced transformation} $T_Y:Y \to Y$ by
\[ T_Y (y) = T^{\varphi_Y(y)}(y).\]
The following result can be found in e.g.~\cite[Proposition 1.5.7]{Aar97}. Note that this statement asks for $T$ to be conservative. This is not used in the proof however and the condition $m \big( X \setminus \bigcup_{n \ge 1} T^{-n} Y \big)=0$ is enough to guarantee that the induced transformation is well defined.

\begin{lemma}[see e.g.~Proposition 1.5.7.~in \cite{Aar97}]\label{l:inducedmeasure}
Let $T$ be a measurable and non-singular transformation on a measure space $(X, \mathcal F, m)$ and let $Y \in \mathcal F$ be such that $0 < m(Y) < \infty$ and $m \big(X \setminus \bigcup_{n \ge 1} T^{-n}Y\big)=0$. If $\nu \ll m|_Y$ is a finite invariant measure for the induced transformation $T_Y$, then the measure $\mu$ on $(X, \mathcal F, m)$ defined by
\[ \mu (B) = \sum_{k \ge 0} \nu \Big(Y \cap T^{-k}B \setminus \bigcup_{j=1}^k T^{-j}Y \Big)\]
for $B \in \mathcal F$ is $T$-invariant, absolutely continuous with respect to $m$ and $\mu|_Y = \nu$.
\end{lemma}

We will also use the following result on the first return time.
\begin{lemma}[Kac's Formula, see e.g.~1.5.5.~in \cite{Aar97}]\label{l:kac}
Let $T$ be a conservative, ergodic, measure preserving transformation on a measure space $(X, \mathcal F, m)$. Let $Y \in \mathcal F$ be such that $0 < m(Y) < \infty$ and let $\varphi_Y$ be the first return map to $Y$. Then $\int_Y \varphi_Y \, d m = m(X)$.
\end{lemma}

One can also obtain invariant measures via a functional analytic approach. Here we give a specific result for interval maps. Let $I$ be an interval. If $T: I \to I$ is piecewise strictly monotone and $C^1$,
then the {\em Perron-Frobenius operator} $\mathcal P_T$ is defined on the space of non-negative measurable functions $h$ on $I$ by
\begin{equation}\label{q:pfd}
\mathcal P_T h (x) = \sum_{y \in T^{-1}\{x\}} \frac{h(y)}{|DT(y)|}.
\end{equation}
A non-negative measurable function $\varphi$ on $I$ is a fixed point of $\mathcal P_T$ if and only if it provides an invariant measure $\mu$ for $T$ that is absolutely continuous with respect to $\lambda$ by setting $\mu(A) = \int_A \varphi \, d\lambda$ for each Borel set $A$.

\medskip
For a random map $F$ using a finite family of transformations $\{ T_j: I \to I\}_{j \in\Sigma}$, such that each map $T_j$ is piecewise strictly monotone and $C^1$, and a positive probability vector $\mathbf p = (p_j)_{j \in \Sigma}$, the Perron-Frobenius operator $\mathcal{P}_F$ is given on the space of non-negative measurable functions $h$ on $I$ by
\begin{align}\label{eqn3.22}
\mathcal{P}_Fh(x) = \sum_{j \in \Sigma} p_j \mathcal P_{T_j} h (x),
\end{align}
where each $\mathcal{P}_{T_j}$ is as given in \eqref{q:pfd}. Let $\mathbb P$ denote the $\mathbf p$-Bernoulli measure on $\Omega$. Then a non-negative measurable function $\varphi$ on $I$ is a fixed point of $\mathcal P_F $ if and only if the measure $\mathbb P \times \mu$, where $\mu$ is the absolutely continuous measure with $\frac{d\mu}{d\lambda} = \varphi$, is $F$-invariant.

In Subsection \ref{subsec:3.3} it will be shown that the density $\frac{d\mu_{\mathbf p}}{d\lambda}$ from Theorem \ref{MAIN}, which is a fixed point of the Perron-Frobenius operator for the random system $F$ given by \eqref{q:skewproduct}, is bounded away from zero. From this it is easy to see that \eqref{eqn3.22} implies that $\frac{d\mu_{\mathbf p}}{d\lambda}$ blows up to infinity at the points zero and one and also at least on one side of $c$.

\subsection{Estimates on good and bad maps}
Now let $T:I \to I$ be a $C^3$ map of an interval $I$ into itself. The {\em Schwarzian derivative} of $T$ at $x \in I$ with $DT(x) \neq 0$ is defined by
\begin{equation}
\mathbf ST(x) = \frac{D^3T(x)}{DT(x)} - \frac{3}{2} \Big(\frac{D^2 T(x)}{DT(x)}\Big)^2.
\end{equation}
We say that $T$ {\em has non-positive Schwarzian derivative} on $I$ if $DT(x) \neq 0$ and $\mathbf ST(x) \leq 0$ for all $x \in I$.  A direct computation shows that the Schwarzian derivative of the composition of two transformations $T_1,T_2: I \rightarrow I$ satisfies
\begin{equation} \label{eq4}
\mathbf S(T_2 \circ T_1)(x) = \mathbf S T_2\big(T_1(x)\big) \cdot |DT_1(x)|^2 + \mathbf ST_1(x).
\end{equation}
Hence, $\mathbf S(T_2 \circ T_1)\le 0$ provided $\mathbf S T_1\le 0$ and $\mathbf S T_2\le 0$.

\medskip
From \eqref{eq4} it follows that for a finite collection $\{ T_j: I \to I\}_{j \in \Sigma}$ of $C^3$ interval maps with non-positive Schwarzian derivative, we can write the Schwarzian derivative of $T_\omega^n$, $n \in \mathbb{N}$ and $\omega \in \Omega$, as
\begin{align}\label{eqn5}
\mathbf ST_{\omega}^n(x) = \sum_{i=0}^{n-1} \mathbf ST_{\omega_{i+1}}\big(T_{\omega}^i(x)\big) \cdot \Big|\prod_{j=1}^i DT_{\omega_j}(T_{\omega}^{j-1}(x))\Big|^2.
\end{align}
By (G2) and (B2) this implies that for a collection of good and bad maps $\{ T_j\}_{j \in \Sigma}$, $T_{\omega}^n$ has non-positive Schwarzian derivative on $[0,1]$ outside of the critical points of $T_\omega^n$ for all $\omega\in \Omega$ and $n \in \mathbb{N}$. 

\medskip
We will use the following two well-known properties of maps with non-positive Schwarzian derivative (see e.g.~\cite[Section 4.1]{dMvS93}).
\vskip .1cm
\noindent {\bf Minimum Principle:} Let $I = [a,b]$ be a closed interval and suppose that $T: I \rightarrow I$ has non-positive Schwarzian derivative. Then
\begin{equation}
|DT(x)| \geq \min\{DT(a),DT(b)\}, \quad \forall x \in [a,b].
\end{equation}

\medskip
A consequence of the Minimum Principle is that for any $T \in \mathfrak G \cup \mathfrak B$ the derivative $|DT|$ has locally no strict minima in the intervals $(0,c)$ and $(c,1)$. In particular, there cannot be any attracting fixed points for $T$ in $(0,c)$ and $(c,1)$. Therefore, if $T \in \mathfrak B$, then $T^n(x) \rightarrow c$ as $n \rightarrow \infty$ for all $x \in (0,1)$.

\medskip
\noindent {\bf Koebe Principle:} For each $\rho > 0$ there exist $K^{(\rho)} >1$ and $M^{(\rho)} > 0$ with the following property. Let $J \subseteq I$ be two intervals and suppose that $T: I \rightarrow I$ has non-positive Schwarzian derivative. If both components of $T(I)\backslash T(J)$ have length at least $\rho \cdot \lambda(T(J))$, then
\begin{align}\label{eq2.4}
\frac{1}{K^{(\rho)}} \leq \frac{DT(x)}{DT(y)} \leq K^{(\rho)}, \qquad \forall x,y \in J
\end{align}
and
\begin{align}\label{eq2.5}
\Big|\frac{DT(x)}{DT(y)}-1\Big| \leq M^{(\rho)} \cdot \frac{|T(x)-T(y)|}{\lambda(T(J))}, \qquad \forall x,y \in J.
\end{align}
Note that the constants $K^{(\rho)},M^{(\rho)}$ only depend on $\rho$ and not on the map $T$.

\medskip
From \eqref{eq2.4} one can obtain a bound on the size of the images of intervals: Let $J' \subseteq J$ be another interval. By the Mean Value Theorem there exists an $x \in J'$ with $|DT(x)| = \frac{\lambda (T(J'))}{\lambda(J')}$ and a $y \in J$ with $|DT(y)| = \frac{\lambda (T(J))}{\lambda(J)}$. Hence,
\begin{equation}\label{q:intervals}
\frac1{K^{(\rho)}} \frac{\lambda(J')}{\lambda(J)} \le \frac{DT(x)}{DT(y)} \frac{\lambda(J')}{\lambda(J)}  = \frac{\lambda(T(J'))}{\lambda(T(J))} \le K^{(\rho)} \frac{\lambda(J')}{\lambda(J)}.
\end{equation}

\medskip
Recall the constants $\ell_b$, $K_b$ and $M_b$ from condition (B3) and set $\ell_{\min} = \min \{ \ell_b \, : \, b \in \Sigma_B\}$ and $\ell_{\max} = \max \{ \ell_b \, : \, b \in \Sigma_B\}$. (B3) gives us control over the distance between $T_\omega^n (x)$ and $c$. 

\begin{lemma}\label{lemma3.6}
For all $n \in \mathbb{N}$, $\omega \in \Sigma_B^{\mathbb{N}}$ and $x \in [0,1]$,
\begin{align}
\left( \tilde K |x-c| \right)^{\ell_{\omega_1} \cdots \ell_{\omega_n}}  \le  |T_{\omega}^n(x)-c| \leq \left( \tilde M |x-c|\right)^{\ell_{\omega_1} \cdots \ell_{\omega_n}},
\end{align}
with constants $\tilde K =\big( \frac{\min \{ K_b \, : \, b \in \Sigma_B \}}{\ell_{\max}}\big)^ \frac1{\ell_{\min}-1} \in (0,1)$ and $\tilde M = \big(\frac{\max\{ M_b \, : \, b \in \Sigma_B \}}{\ell_{\min}}\big)^{\frac1{\ell_{\min}-1}} > 1$.
\end{lemma}

\begin{proof}
It follows from (B3) that for any $j \in \Sigma_B$ and $x \in [0,1]$,
\[ |T_j(x)-c| = |T_j(x)-T_j(c)| = \Big| \int_c^x DT_j(y) dy \Big| \ge \frac{\min \{ K_b \, : \, b \in \Sigma_B \}}{\ell_{\max}} |x-c|^{\ell_j}.\]
By induction we get that for each $n \in \mathbb{N}$ and $\omega \in \Sigma_B^{\mathbb{N}}$,
\begin{align}\label{eq3.29}
|T_{\omega}^n(x)-c| \ge \left(  \frac{\min \{ K_b \, : \, b \in \Sigma_B \}}{\ell_{\max}} \right)^{1+\sum_{i=0}^{n-2} \ell_{\omega_n} \cdots \ell_{\omega_{n-i}}} \cdot |x-c|^{\ell_{\omega_1} \cdots \ell_{\omega_n}}.
\end{align}
From (B3) we see that $\frac{\min \{ K_b \, : \, b \in \Sigma_B \}}{\ell_{\max}}<1$. The lower bound now follows by observing that
\[ \Big(1+\sum_{i=0}^{n-2} \ell_{\omega_n} \cdots \ell_{\omega_{n-i}}\Big)/(\ell_{\omega_1}\cdots \ell_{\omega_n}) \le \sum_{i=1}^n \frac1{\ell_{\min}^i} < \frac1{\ell_{\min}-1}.
\]
The result for the upper bound follows similarly, by noticing that in this case from (B3) it follows that $ \frac{\max\{ M_b \, : \, b \in \Sigma_B\}}{\ell_{\min}}>1$.
\end{proof}

It follows that under iterations of bad maps the distance $|T^n_\omega(x)-c|$ is eventually decreasing superexponentially fast in $n$.

Furthermore, note that there exists a $\delta > 0$ such that $|DT_b(x)| < 1$ for all $x \in [c-\delta,c+\delta]$ and $b \in \Sigma_B$. This implies
\begin{align}\label{q:distancetoc}
|T_b(x)-c| < |x-c|
\end{align}
for all $x \in [c-\delta,c+\delta]$ and $b \in \Sigma_B$.

\medskip
The upper bound on $|T^n_\omega(x)-c|$ that we obtained in Lemma \ref{lemma3.6} will be used in Section \ref{s:in-finite} to prove that $\mu_{\mathbf p}$ in Theorem \ref{MAIN2} is infinite if $\theta \geq 1$. The lower bound from Lemma \ref{lemma3.6} will be used for the proof that $\mu_{\mathbf p}$ is finite if $\theta < 1$.

\section{Existence of a $\sigma$-finite acs measure}\label{sec:PT2.1}
From now on we fix an integer $N \ge 2$ and consider a finite collection $T_1,\ldots,T_N \in \mathfrak G \cup \mathfrak B$ of good and bad maps in the classes $\mathfrak G$ and $\mathfrak B$. As in the Introduction write $\Sigma_G = \{1 \leq j \leq N: T_j \in \mathfrak G\}$ and $\Sigma_B = \{1 \leq j \leq N: T_j \in \mathfrak B\}$ for the corresponding index sets and assume that $\Sigma_G,\Sigma_B \neq \emptyset$. Write $\Sigma = \{ 1, 2, \ldots, N \}$ and set $\Omega = \Sigma^\mathbb N$ for the set of infinite sequences of elements in $\Sigma$. In this section we prove Theorem~\ref{MAIN}, i.e., we establish the existence of an ergodic acs measure and several of its properties  using an inducing scheme for the random system $F$. We fix the index $g \in \Sigma_G$ of one good map $T_g$ and start by constructing an inducing domain that depends on this $g$.

\subsection{The induced system and return time partition}

The first lemma is needed to specify the set on which we induce. For each $k \in \mathbb{N}$ let $x_k$ and $x_k'$ in $(0,c)$ denote the critical points of $T_g^k$ closest to $0$ and $c$, respectively. Furthermore, let $y_k$ and $y_k'$ in $(c,1)$ denote the critical points of $T_g^k$ closest to $1$ and $c$, respectively.

\begin{lemma}\label{l:criticalpts}
We have $x_k \downarrow 0$, $x_k' \uparrow c$, $y_k' \downarrow c$, $y_k \uparrow 1$ as $k \rightarrow \infty$.
\end{lemma}

\begin{proof}
Let $a$ and $b$ denote the critical points of $T^2_g$ in $(0,c)$ and $(c,1)$, respectively. Then at least one of the branches $T_g^2|_{(0,a)}$ and $T_g^2|_{(b,1)}$ is increasing. Suppose that $T_g^2|_{(0,a)}$ is increasing. It then follows from the Minimum Principle that $T_g^2(x) \geq \min\{\frac{x}{a},DT_g^2(0)\cdot x\}$ for each $x \in [0,a]$. To see this, suppose there is an $x \in (0,a)$ with $T_g^2(x) < \min\{\frac{x}{a},DT_g^2(0)\cdot x\}$. Then there must be a $y \in (0, x)$ with $DT_g^2(y) < \min\{ D T_g^2(0), \frac1a \}$ and a $z \in [x,a)$ with $DT_g^2(z) > \frac1a$. On the other hand, by the Minimum Principle, $DT_g^2(y) \ge \min\{ DT_g^2(0), DT_g^2(z) \}$, a contradiction. Combining this with $DT_g^2(0) > 1$ and defining $L: (0,1) \rightarrow (0,a)$ by $L = (T_g^2|_{(0,a)})^{-1}$, we see that $L^k(a) \downarrow 0$ as $k \rightarrow \infty$. Furthermore, define $R: (0,1) \rightarrow (b,1)$ by $R = (T_g^2|_{(b,1)})^{-1}$. If $T_g^2|_{(b,1)}$ is increasing, we see that similarly $R^k(b) \uparrow 1$ as $k \rightarrow \infty$. On the other hand, if $T_g^2|_{(b,1)}$ is decreasing, we have $RL^k(a) \uparrow 1$ as $k \rightarrow \infty$. Finally, if $T_g^2|_{(0,a)}$ is decreasing, then $T_g^2|_{(b,1)}$ must be increasing, which yields $LR^k(b) \downarrow 0$ as $k \rightarrow \infty$. We conclude that $x_k \downarrow 0$ and $y_k \uparrow 1$ as $k \rightarrow \infty$. It follows from (G1) that $c$ is a limit point of both of the sets $\bigcup_{k \in \mathbb{N}} (T_g|_{(0,c)})^{-1}(\{x_k,y_k\})$ and $\bigcup_{k \in \mathbb{N}} (T_g|_{(c,1)})^{-1}(\{x_k,y_k\})$. So $x_k' \uparrow c$, $y_k' \downarrow c$ as $k \rightarrow \infty$. 
\end{proof}

By the previous lemma and (G1), for $k \in \mathbb{N}$ large enough it holds that
\begin{equation}\label{eq:3.1}
\begin{split}
T_g(x_k') \leq \ &  x_k' \text{ or } T_g(x_k') \geq y_k', \text{ and}\\
T_g(y_k') \leq \ & x_k' \text{ or } T_g(y_k') \geq y_k',
\end{split}\end{equation}
and, using also (G4), (B1) and (B4), for every $j \in \Sigma$,
\begin{equation}\label{eq:3.3}
\begin{split}
T_j & \big([0,x_k] \cup [y_k,1]\big) \subseteq  [0,x_k') \cup (y_k',1] \text{ and }\\
 |D & T_j(x)|> d > 1 \, \,  \text{ for all } x \in [0,x_k) \cup (y_k,1] \, \text{ and some constant }d.
\end{split}
\end{equation}
Fix a $\kappa \in \mathbb N$ for which \eqref{eq:3.1} and \eqref{eq:3.3} hold. We introduce some notation. Let $t \in \Sigma$ be such that $t \neq g$, and define
\begin{align}
&C = [\underbrace{g\cdots g}_{\kappa \ \text{times}} t] = [g^\kappa t],\\
&J_0 = (x_\kappa,x_\kappa'), \quad J_1 = (y_\kappa',y_\kappa), \quad J = J_0 \cup J_1,\\
&Y = C \times J.
\end{align}
The next lemma shows that $\mathbb P \times \lambda$-almost all $(\omega,x)$ eventually enter $Y$ under iterations of $F$, 
and hence that $\mathbb P \times \lambda$-almost all $(\omega,x) \in Y$ will return to $Y$ infinitely many times. 
\begin{lemma}\label{lemma3.2}

\begin{align}
\mathbb P \times \lambda \Big( \Omega \times [0,1] \setminus \bigcup_{n = 1}^{\infty} F^{-n} Y \Big)=0.
\end{align}
\end{lemma}

\begin{proof} For $\mathbb{P}$-almost all $\omega \in \Omega$ we have $\sigma^n \omega \in [g]$ for infinitely many $n \in \mathbb{N}$. For any such $n$ and each $x \in (0,c) \cup (c,1)$ either $T_\omega^n (x) \in J$ or $T_\omega^n(x) \not \in J$. If $T_\omega^n(x) \in (0, x_\kappa] \cup [y_\kappa,1)$, then it follows from \eqref{eq:3.3} that there is an $m \ge 1$ such that $T_\omega^{n+m}(x) \in J$. If $T_\omega^n (x) \in [x_\kappa',c) \cup (c, y_\kappa']$ it follows from \eqref{eq:3.1} that $T^{n+1}_\omega (x)=T_g \circ T_\omega^n (x) \in (0,x_\kappa'] \cup [y_\kappa',1)$, which means that we are in the first case if $T_{\omega}^{n+1}(x) \notin J$. Hence, for $\mathbb{P} \times \lambda$-almost all $(\omega,x) \in \Omega \times [0,1]$ we have $T_{\omega}^n(x) \in J$ for infinitely many $n \in \mathbb{N}$. Consider such an $(\omega,x)$, and let $(n_j)_{j \in \mathbb{N}}$ be an increasing sequence in $\mathbb{N}$ that satisfies $T_{\omega}^{n_j}(x) \in J$ for each $j \in \mathbb{N}$. Recall that $\sigma$ denotes the left shift on sequences and define $\mathcal{E} = \Sigma \backslash (\bigcup_{j \in \mathbb{N}} \sigma^{-n_j} C)$, i.e., $\mathcal E$ contains precisely those sequences $\omega'$ that satisfy $\sigma^{n_j}(\omega') \not \in C$ for all $n_j$. According to the Lebesgue Differentiation Theorem (see e.g.~\cite{Tol03}) we may assume that $\omega$ is a Lebesgue point of $1_{\mathcal{E}}$, which yields
\begin{align*}
1 \ge \frac{\mathbb{P}\big((\mathcal{E} \cup \sigma^{-n_j}C) \cap [\omega_1 \cdots \omega_{n_j}]\big)}{\mathbb{P}\big([\omega_1 \cdots \omega_{n_j}]\big)} &= \frac{\mathbb{P}\big(\mathcal{E} \cap [\omega_1 \cdots \omega_{n_j}]\big)}{\mathbb{P}\big([\omega_1 \cdots \omega_{n_j}]\big)} + \frac{\mathbb{P}\big(\sigma^{-n_j}C \cap [\omega_1 \cdots \omega_{n_j}]\big)}{\mathbb{P}\big([\omega_1 \cdots \omega_{n_j}]\big)}\\
& \rightarrow 1_{\mathcal{E}}(\omega) + \mathbb{P}(C), \qquad \text{ as $j \rightarrow \infty$}.
\end{align*}
Since $\mathbb{P}(C) > 0$,  we conclude that $\omega \notin \mathcal{E}$. Hence, there is an $n_j$ so that $F^{n_j}(\omega,x) \in C \times J=Y$.
\end{proof}

By Lemma~\ref{lemma3.2} the first return time map $\varphi_Y$, see \eqref{q:frtm}, and the induced transformation $F_Y$ are well defined on the full measure subset of points in $Y$ that return to $Y$ infinitely often under iterations of $F$, which we call $Y$ again.
The set of points in $Y$ that return to $Y$ after $n$ iterations of $F$ can be described as
\begin{equation}\label{q:leveln}
Y \cap F^{-n}(Y) = \bigcup_{\omega \in C \cap \sigma^{-n}C} [\omega_1 \cdots \omega_n] \times (T_{\omega}^n|_J)^{-1}(J) \quad \bmod \mathbb{P} \times \lambda,
\end{equation}
which is empty for $n \le \kappa$. Note that in \eqref{q:leveln} in fact $[\omega_1 \cdots \omega_n] = [g^\kappa t \omega_{k+2}\cdots \omega_n g^\kappa t]$ and that by construction each map $T_{\omega}^n|_J$ in \eqref{q:leveln} consists of branches that all have range $(0,c)$ or $(c,1)$ or $(0,1)$, since any branch of $T_{\omega}^\kappa|_J$ maps onto $(0,1)$. Therefore, $Y \cap F^{-n}(Y)$ can be written as a finite union of products $A = [\mathbf u g^{\kappa} t] \times I$ of cylinders $[\mathbf u g^{\kappa} t] \subseteq C$ with $|\mathbf u|=n$ and open intervals $I \subseteq J$, each of which is mapped under $F^n$ onto $C \times J_0$ or $C \times J_1$. Call the collection of these sets $P_n$ and let $\alpha = \bigcup_{n > \kappa} P_n$. Let $\mathbb{P}_C$ and $\lambda_J$ denote the normalized restrictions of $\mathbb{P}$ to $C$ and $\lambda$ to $J$ respectively.

\begin{lemma}\label{l:disjointp}
\leavevmode
\begin{itemize}
\item[(1)] The collection $\alpha$ forms a countable {\em return time partition} of $Y$, i.e., the measure $\mathbb P_C \times \lambda_J (\bigcup_{A \in \alpha} A) = 1$, any two different sets $A, A' \in \alpha$ are disjoint and on any $A \in \alpha$ the first return time map $\varphi_Y$ is constant.
\item[(2)] Let $\pi$ denote the canonical projection onto the second coordinate. Any $x \in J$ is contained in a set $\pi(A)$ for some set $A \in \alpha$.
\end{itemize}
\end{lemma}

\begin{proof} 
The fact that $\mathbb P_C \times \lambda_J (\bigcup_{A \in \alpha} A) = 1$ follows from Lemma~\ref{lemma3.2} and it is clear from the construction that the first return time map $\varphi_Y$ is constant on any element $A \in \alpha$. To show that any two elements are disjoint, note that for $A,A' \in P_n$ this is clear. Suppose there are $1 \le m <n$, $A = [\mathbf u g^{\kappa} t] \times I \in P_n$ and $A' = [\mathbf v g^{\kappa} t] \times I' \in P_m$ such that $A \cap A' \neq \emptyset$. Since $t \neq g$ we get $n \geq m+\kappa+1$ and $[\mathbf u g^\kappa t] = [g^\kappa t v_{\kappa+2}\cdots v_m g^\kappa t u_{m+\kappa+2} \cdots u_n g^\kappa t]$. Moreover, $I \cap \partial I' \neq \emptyset$ or $I=I'$. In both cases, note that $F^{m+\kappa+1} ([\mathbf v g^{\kappa} t] \times \partial I') \subseteq \Omega \times \{0,1\}$, so by (G1) and (B1) also $F^n([\mathbf v g^{\kappa} t] \times \partial I') \subseteq \Omega \times \{0,1\}$, contradicting that $F^n(A) \subseteq Y$. This proves (1).

\medskip
For (2) note that, since $\alpha$ is a partition of $Y$, for each $x \in J$ it holds that there is an $A = [\mathbf u g^{\kappa} t] \times I \in \alpha$ with $x \in I$ or $x \in \partial I$. In the first case there is nothing to prove, so assume that $x \in \partial I$. Then $T_{\mathbf u}(x) \in \partial J_i$ for some $i \in \{0,1\}$. From the first part of the proof of Lemma~\ref{lemma3.2} it then follows that there is an $n > |\mathbf u|$ and an $\omega \in C$ such that $T^n_\omega(x) \in J$. If we write $I'$ for the interval in $T^{-n}_\omega (J)$ containing $x$, then this means that there exists a set $A'=[\mathbf vg^{\kappa} t] \times I' \in \alpha$ with $x \in \pi(A')$. 
\end{proof}

The second part of Lemma~\ref{l:disjointp} shows that even though the partition elements of $\alpha$ are disjoint, their projections on the second coordinate are not. The same is true for the first coordinate as the same string $\mathbf u$ can lead points in $J$ to $J_0$ and $J_1$.

\subsection{Properties of the induced transformation}
It follows from \eqref{q:leveln} and Lemma~\ref{l:disjointp} that for each $A \in \alpha$ we have either $F_Y(A) = C \times J_0$ or $F_Y(A) = C \times J_1$. For any $[\mathbf ug^{\kappa} t] \times I \in \alpha$, the transformation $T_{\mathbf u}|_I$ is invertible from $I$ to one of the sets $J_0$ or $J_1$. Define the operator $\mathcal{P}_{\mathbf u,I}: L^1(J,\lambda_J) \rightarrow L^1(J,\lambda_J)$ by
\begin{align}
\mathcal{P}_{\mathbf u,I}h(x) = \begin{cases}
\displaystyle  \frac{h(T_\mathbf u|_I^{-1}(x))}{\big|DT_\mathbf u|_I(T_\mathbf u|_I^{-1}(x))\big|}, & \text{if } T_\mathbf u|_I^{-1}\{x\} \neq \emptyset,\\
0, & \text{otherwise}.
\end{cases}
\end{align}
The {\em random Perron-Frobenius-type operator} $\mathcal{P}_Y: L^1(J,\lambda_J) \rightarrow L^1(J,\lambda_J)$ on $Y$ is given by
\begin{align}\label{eqn3.12}
\mathcal{P}_Y = \sum_{[\mathbf u g^{\kappa} t] \times I \in \alpha} \mathbb{P}_C([\mathbf u]) \mathcal{P}_{\mathbf u,I}.
\end{align}
Note that $\mathcal P_Y$ is not exactly of the same form as the usual Perron-Frobenius operator in \eqref{eqn3.22}. Nonetheless, we have the following result.

\begin{lemma}\label{l:pflike}
If $\varphi \in L^1(J,\lambda_J)$ is a fixed point of $\mathcal{P}_Y$, then the measure $\mathbb{P}_C \times \nu$ with $\nu = \varphi d\lambda_J$ is invariant for $F_Y$. 
\end{lemma}

\begin{proof}
For each cylinder $K \subseteq C$ and each Borel set $E \subseteq J$ we have
\begin{align*}
\mathbb{P}_C \times \nu\big(F_Y^{-1}(K \times E)\big) 
&= \sum_{[\mathbf u g^{\kappa} t] \times I \in \alpha} \mathbb{P}_C([\mathbf u g^{\kappa} t] \cap \sigma^{-|\mathbf u|} K) \nu(I \cap T_\mathbf u^{-1} E)\\
& = \mathbb{P}_C(K) \sum_{[\mathbf u g^{\kappa} t] \times I \in \alpha} \mathbb{P}_C([\mathbf u]) \int_E \mathcal{P}_{\mathbf u,I} \varphi d\lambda_J\\
&= \mathbb{P}_C(K) \int_E \mathcal{P}_Y \varphi d\lambda_J\\
&= \mathbb{P}_C \times \nu(K \times E). \qedhere
\end{align*}
\end{proof}

\medskip
In Lemma~\ref{lemma3.4} below we show that a fixed point of $\mathcal P_Y$ exists. For $m \in \mathbb{N}$, set $\alpha_m = \bigvee_{j=0}^{m-1} F_Y^{-j} \alpha$. Atoms of this partition are the {\em $m$-cylinders of $F_Y$}. Introducing for each $Z = \bigcap_{j=0}^{m-1} F_Y^{-j} ([\mathbf u_j g^{\kappa} t] \times I_j)$ in $\alpha_m$ the notation
\begin{equation}\label{q:czjz}
C_Z = \bigcap_{j=0}^{m-1} \sigma^{-\sum_{i=0}^{j-1}|\mathbf u_i|}[\mathbf u_j g^{\kappa} t] \quad \text{ and } \quad J_Z = \bigcap_{j =0}^{m-1} T_{\mathbf u_0 \mathbf u_1\cdots \mathbf u_{j-1}}^{-1} (I_j),
\end{equation}
we obtain $Z = C_Z \times J_Z$. Writing $\sigma_Z = \sigma^{\sum_{i=0}^{m-1}|\mathbf u_i|}|_{C_Z}$ and $T_Z = T_{\mathbf u_0 \mathbf u_1\cdots \mathbf u_{m-1}}|_{J_Z}$ we have $F_Y^m|_Z = \sigma_Z \times T_Z$. 
Each $T_Z $ has non-positive Schwarzian derivative, so we can apply the Koebe Principle. The image $T_Z(J_Z)$ either equals $J_0$ or $J_1$. Choose a $\bar \rho >0$ such that $I_0:=[x_\kappa - \bar \rho, x_\kappa' + \bar \rho] \subseteq (0,c)$ and $I_1:=[y_\kappa'-\bar \rho, y_\kappa + \bar \rho] \subseteq (c,1)$. There is a canonical way to extend the domain of each $T_Z$ to an interval $I$ containing $J_Z$, such that $T_Z(I)$ equals either $I_0$ or $I_1$ and $\mathbf S(T_Z) \le 0$ on $I$. Then by the Koebe Principle there exist constants $K^{(\bar \rho)}>1$ and $M^{(\bar \rho)} > 0$ such that for all $m \in \mathbb{N}$, $Z \in \alpha_m$ and $x,y \in J_Z$,
\begin{align}\label{eqn3.16a}
\frac{1}{K^{(\bar \rho)}} \leq \frac{DT_Z(x)}{DT_Z(y)} \leq K^{(\bar \rho)},
\end{align}
\begin{align}\label{eqn3.16}
\Big|\frac{DT_Z(x)}{DT_Z(y)}-1\Big| \leq \frac{M^{(\bar \rho)}}{\min\{ \lambda(I_0), \lambda(I_1)\}} \cdot |T_Z(x)-T_Z(y)|.
\end{align}
Note that for the random Perron-Frobenius-type operator from \eqref{eqn3.12} we have for each $m \geq 1$ that
\begin{align}
\mathcal{P}_Y^m = \frac{1}{\mathbb{P}(C)} \sum_{Z \in \alpha_m} \mathbb{P}_C(C_Z) \mathcal{P}_{T_Z},
\end{align}
where $\mathcal{P}_{T_Z}$ is as in \eqref{q:pfd}.

\begin{lemma}[cf.~Lemmata V.2.1 and V.2.2 of \cite{dMvS93}]\label{lemma3.4}
$\mathcal{P}_Y$ admits a fixed point $\varphi \in L^1(J,\lambda_J)$ that is bounded, Lipschitz and bounded away from zero.
\end{lemma}

\begin{proof} For each $m \in \mathbb{N}$ and $x \in J$,
\begin{align}
\mathcal{P}_Y^m 1(x) = \frac{1}{\mathbb{P}(C)} \sum_{\stackrel{Z \in \alpha_m:}{ x \in T_Z(J_Z)}} \frac{\mathbb{P}_C(C_Z)}{|DT_Z(T_Z^{-1}x)|}.
\end{align}
Using the Mean Value Theorem, for all $m \in \mathbb{N}$ and $Z \in \alpha_m$ there exists a $\xi \in J_Z$ such that
\begin{equation}\label{q:mvt}
\frac{\lambda\big(T_Z(J_Z)\big)}{\lambda(J_Z)} = |DT_Z(\xi)|.
\end{equation}
Set $K_1 = \frac{\max \{K^{(\bar \rho)}, M^{(\bar \rho)}\}}{\mathbb{P}(C)\cdot \min\{\lambda(J_0),\lambda(J_1)\}}$, where $\bar \rho$ is as in  \eqref{eqn3.16a} and \eqref{eqn3.16}. Since $DT_Z(\xi)$ and $DT_Z(y)$ have the same sign for any $y \in J_Z$, \eqref{q:mvt} together with \eqref{eqn3.16a} implies
\begin{equation}\label{q:bounded}
\mathcal{P}_Y^m 1(x) \le \sum_{Z \in \alpha_m} \frac{\mathbb{P}_C(C_Z)}{\mathbb{P}(C)} \cdot K^{(\bar \rho)} \frac{\lambda (J_Z)}{\lambda(T_Z(J_Z))} \le  K_1 \sum_{Z \in \alpha_m} \mathbb P_C \times \lambda_J (C_Z \times J_Z) = K_1.
\end{equation}
Moreover, if for $A = [\mathbf u g^{\kappa} t] \times I \in \alpha$ we take $x,y \in I$, then for any $Z \in \alpha_m$ it holds that $x \in T_Z(J_Z)$ if and only if $y \in T_Z(J_Z)$. For such $Z$, let $x_Z, y_Z \in J_Z$ be such that $T_Z(x_Z)=x$ and $T_Z (y_Z)=y$. Then by \eqref{eqn3.16}
\begin{equation}\label{q:equict} \begin{split}
|\mathcal P^m_Y 1(x) - \mathcal P^m_Y 1(y)| \le \ & \sum_{\stackrel{Z \in \alpha_m:}{ x \in T_Z(J_Z)}} \frac{\mathbb P_C(C_Z)}{\mathbb{P}(C)} \left| \frac{1}{|D T_Z (x_Z)|} - \frac{1}{|D T_Z (y_Z)|}  \right|\\
\le \ & \sum_{\stackrel{Z \in \alpha_m:}{ x \in T_Z(J_Z)}} \mathbb P_C(C_Z) \frac{1}{|D T_Z (x_Z)|} K_1 |T_Z (x_Z)- T_Z (y_Z)|\\
=\ & K_1 \mathcal P_Y^m 1 (x) |x-y|.
\end{split}
\end{equation}
Together \eqref{q:bounded} and \eqref{q:equict} imply that the sequence $\big( \frac1m \sum_{j=0}^{m-1} \mathcal P^j_Y 1 \big)_m$ is uniformly bounded and equicontinuous on $I$ for each $A = [\mathbf u g^{\kappa} t] \times I$. By Lemma~\ref{l:disjointp}(2) it follows that the same holds on $J$. Hence, by the Arzela-Ascoli Theorem there exists a subsequence
\[ \left( \frac1{m_k} \sum_{j=0}^{m_k-1} \mathcal P^j_Y 1 \right)_{m_k}\]
converging uniformly to a function $\varphi:J \to [0,\infty)$ satisfying $\varphi \le K_1$ and for each $A = [\mathbf u g^{\kappa} t] \times I \in \alpha$ and $x,y \in I$,
\begin{equation}\label{q:varphieqct}
|\varphi(x)-\varphi(y)| \le K_1 \varphi(x) |x-y|.
\end{equation}

Hence, $\varphi$ is bounded and by Lemma~\ref{l:disjointp}(2) it is clear that $\varphi$ is Lipschitz (with Lipschitz constant bounded by $K_1^2$). It is readily checked that $\varphi$ is a fixed point of $\mathcal P_Y$, so that $\mathbb P_C \times \nu$ with $\nu = \varphi \, d\lambda$ is an invariant probability measure for $F_Y$.

\medskip
What is left is to verify that for each $A = [\mathbf ug^{\kappa}t] \times I \in \alpha$ the function $\varphi$ is bounded from below on the interior of $I$. Suppose that there is such an $A = [\mathbf ug^{\kappa}t] \times I$ for which $\inf_{x \in I} \varphi(x)=0$. Then from \eqref{q:varphieqct} it follows that $\varphi(y)=0$ for all $y \in I$, hence $\nu(I)=0$. Either $I \subseteq J_0$ or $I \subseteq J_1$. If $I \subseteq J_0$, then for any set $A' = [\mathbf v g^{\kappa} t] \times I'\in \alpha$ with $T_\mathbf v (I') =J_0$ it holds that
\[ \mathbb{P}_C \times \lambda_J(A' \cap F_Y^{-1} A) > 0\]
and, by the $F_Y$-invariance of $\mathbb{P}_C \times \nu$,
\[ \mathbb{P}_C \times \nu(A' \cap F_Y^{-1} A) \leq \mathbb{P}_C \times \nu(F_Y^{-1} A) = \mathbb{P}_C \times \nu(A) = 0,\]
which together give $\inf_{x \in I'} \varphi(x) = 0$ and therefore, like before, $\nu(I') = 0$. There are sets $A' =  [\mathbf v g^{\kappa} t] \times I'$ with $I' \subseteq J_1$ and $T_{\mathbf v}(I') = J_0$, so we can repeat the argument to show that also for any set $A''= [\mathbf v g^{\kappa} t] \times I'' \in \alpha$ with $T_{\mathbf v}(I'')=J_1$ we have $\nu(I'')=0$. So $\mathbb P_C \times \nu(A)=0$ for all $A \in \alpha$. If $I \subseteq J_1$ we come to the same conclusion. This gives a contradiction, so $\varphi$ is bounded from below on each interval $I$.
\end{proof}

It follows from Lemma~\ref{l:pflike} that $\mathbb{P}_C \times \nu$ with $\nu = \varphi d\lambda_J$ is a finite $F_Y$-invariant measure. To show that $\mathbb P_C \times \lambda_J$ is $F_Y$-ergodic we need the following result, which states that the sets $\pi(A)$ for $A \in \alpha_m$ shrink uniformly to $\lambda$-null sets as $m \to \infty$.

\begin{lemma}\label{l:shrinking}
$\displaystyle \lim_{m \to \infty} \sup \{ \lambda_J(J_Z) \, : \, Z \in \alpha_m \} =0$.
\end{lemma}

\begin{proof}
Set $\delta = \sup \{ \lambda_J(J_Z) \, : \, Z \in \alpha\} < 1$. Fix an $m$ and let $Z = \bigcap_{j=0}^{m-1} F_Y^{-j} ([\mathbf u_j g^{\kappa}t] \times I_j) = C_Z \times J_Z \in \alpha_m$ as in \eqref{q:czjz}. Set
\[ \tilde J_Z = \bigcap_{j =0}^{m-2} T_{\mathbf u_0 \mathbf u_1\cdots \mathbf u_{j-1}}^{-1} (I_j),\]
so that $J_Z = \tilde J_Z \cap T_{\mathbf u_0 \cdots \mathbf u_{m-2}}^{-1} (I_{m-1})$. Let $J_i$, $i \in \{0,1\}$, be such that $T_{\mathbf u_0 \cdots \mathbf u_{m-2}} (\tilde J_Z) = J_i$. It holds that $T_{\mathbf u_0 \cdots \mathbf u_{m-2}}(J_Z) = I_{m-1}$, so $\lambda (T_{\mathbf u_0 \cdots \mathbf u_{m-2}}(J_Z)) \le \delta$ and thus 
\[ \lambda( T_{\mathbf u_0 \mathbf u_1 \cdots \mathbf u_{m-2}} ( \tilde J_Z  \setminus J_Z) \ge \lambda(J_i)-\delta.\]
Since $\tilde J_Z \setminus J_Z$ consists of at most two intervals, with \eqref{eqn3.16a} and \eqref{q:intervals} this gives
\[ 1- \frac{\lambda_J(J_Z)}{\lambda_J(\tilde J_Z)} = \frac{\lambda_J (\tilde J_Z\setminus J_Z)}{\lambda_J(\tilde J_Z)} \ge \frac1{K^{(\bar \rho)}} \frac{\lambda_J(T_{\mathbf u_0 \cdots \mathbf u_{m-2}}(\tilde J_Z \setminus J_Z)}{\lambda_J(T_{\mathbf u_0 \cdots \mathbf u_{m-2}}(\tilde J_Z))} \ge \frac1{K^{(\bar \rho)}} \frac{\lambda_J(J_i)-\delta}{\lambda_J(J_i)}.\]
Set $K_1 := \max \big\{ 1 - \frac1{K^{(\bar \rho)}} \frac{\lambda_J(J_i)-\delta}{\lambda_J(J_i)} \, : \, i=0,1 \big\}  \in (0,1)$. Then by repeating the same steps, we obtain
\[ \lambda_J(J_Z)  \le K_1 \lambda_J(\tilde J_Z) \le \cdots \le K_1^m \lambda_J(I_0) < K_1^m,\]
which proves the lemma.
\end{proof}

\begin{lemma}
The measure $\mathbb{P}_C \times \lambda_J$ is $F_Y$-ergodic.
\end{lemma}

\begin{proof} Suppose $E \subseteq Y$ with $\mathbb{P}_C \times \lambda_J(E) > 0$ satisfies $F_Y^{-1} E = E$ mod $\mathbb{P}_C \times \lambda_J$. We show that $\mathbb{P}_C \times \lambda_J(E) =1$. The Borel measure $\rho$ on $Y$ given by
\[ \rho(V) = \int_V 1_E(\omega,x) \varphi(x) d\mathbb{P}_C(\omega)d\lambda_J(x)\]
for Borel sets $V$ is $F_Y$-invariant. According to Lemma~\ref{l:inducedmeasure} and Lemma~\ref{l:productmeasure} this yields a stationary measure $\tilde{\mu}$ on $[0,1]$ that is absolutely continuous w.r.t.~$\lambda$ and satisfies $(\mathbb{P} \times \tilde{\mu})|_Y = \rho$. Let $L := \text{supp}(\tilde{\mu}|_J)$ denote the support of the measure $\tilde{\mu}|_J$. Since $\rho$ is a product measure, this gives $\text{supp}(\rho) = C \times L$ and so by the definition of $\rho$ we get $C \times L \subseteq E$ and $\rho(E \backslash (C \times L)) = 0$. Since $\varphi$ is bounded away from zero, this yields
\begin{align}\label{eqn3.17}
E = C \times L \quad \bmod \mathbb{P}_C \times \lambda_J.
\end{align}
To obtain the result, it remains to show that $\lambda_J(J \backslash L) = 0$. 

\medskip
We have $C \times L = \bigcup_{Z \in \alpha_m} C_Z \times (J_Z \cap L)$ and $F_Y^{-m}(C \times L) = \bigcup_{Z \in \alpha_m} C_Z \times T_Z^{-1} L$. From the non-singularity of $F_Y$ w.r.t.~$\mathbb{P}_C \times \lambda_J$ it follows that for each $m \in \mathbb{N}$,
\begin{align}\label{eqn3.18}
C \times L = E = F_Y^{-m}E = F_Y^{-m}(C \times L) \quad \bmod \mathbb{P}_C \times \lambda_J,
\end{align}
which yields
\begin{align}\label{eqn3.19}
 J_Z \cap L = T_Z^{-1} L \quad \bmod \lambda_J, \quad \text{ for each $Z \in \alpha_m$}.
\end{align}
Let $\varepsilon >0$. Since $\lambda_J(L) > 0$, it follows from Lemma~\ref{l:shrinking} and the Lebesgue Density Theorem that there are $i\in\{0,1\}$, $m_i \in \mathbb{N}$ and $Z_i \in \alpha_{m_i}$ such that
\[ T_{Z_i}(J_{Z_i}) = J_i \quad \text{ and } \quad \lambda_J(J_{Z_i} \cap L) \geq (1-\varepsilon) \lambda_J(J_{Z_i}).\]
By \eqref{eqn3.19}, $T_{Z_i}^{-1}(J_i\setminus L) = J_{Z_i} \setminus L \bmod \lambda_J$. The Mean Value Theorem gives the existence of a $\xi \in J_{Z_i}$ such that
\[ \frac{\lambda_J(T_{Z_i}(J_{Z_i}))}{\lambda_J(J_{Z_i})} = |D T_{Z_i}(\xi)|,\]
and from \eqref{eqn3.16a} it follows that
\[ \lambda_J (T_{Z_i} (J_{Z_i}  \setminus L)) = \int_{J_{Z_i}  \setminus L} |DT_{Z_i}| d\lambda \le K^{(\bar \rho)} |DT_{Z_i}(\xi)| \lambda_J (J_{Z_i} \setminus L).\]
Hence,
\begin{align}\label{eqn3.23}
\frac{\lambda_J(J_i \backslash L)}{\lambda_J(J_i)} = \frac{\lambda_J(T_{Z_i}(J_{Z_i}  \setminus L))}{\lambda_J(T_{Z_i}(J_{Z_i}))} \leq K^{(\bar \rho)} \frac{\lambda_J(J_{Z_i} \backslash L)}{\lambda_J(J_{Z_i})} \leq K^{(\bar \rho)} \varepsilon.
\end{align}
So, for each $\varepsilon>0$ we can find an $i=i(\varepsilon)$ for which \eqref{eqn3.23} holds. If for each $\varepsilon_0 > 0$ and each $i_0 \in \{0,1\}$ there exists an $\varepsilon \in (0,\varepsilon_0)$ such that $i(\varepsilon) = i_0$, we obtain from \eqref{eqn3.23} that $\lambda_J(J \backslash L) = 0$. Otherwise, there exists $\varepsilon_0 > 0$ and $i_0 \in \{0,1\}$ such that $i(\varepsilon) = i_0$ for all $\varepsilon \in (0,\varepsilon_0)$. Without loss of generality, suppose that $i_0 = 0$. Then \eqref{eqn3.23} gives $\lambda_J(J_0 \backslash L) = 0$. By the equivalence of $\nu$ and $\lambda_J$ and the fact that every good map has full branches it follows that
\begin{align}
\mathbb{P}_C \times \nu\big((C \times J_0) \cap F_Y^{-1}(C \times J_1)\big) > 0.
\end{align}
Together with the Poincar\'e Recurrence Theorem this gives that
\begin{align}
A = \{(\omega,x) \in C \times J_0: F_Y^m(\omega,x) \in C \times J_1 \text{ for infinitely many $m \in \mathbb{N}$}\}
\end{align}
satisfies $\mathbb{P}_C \times \nu(A) > 0$, and therefore $\mathbb{P}_C \times \lambda_J(A) > 0$. Together with $\lambda_J(J_0 \backslash L ) = 0$ it follows from the Lebesgue Density Theorem that there exists a Lebesgue point $x \in \pi(A) \cap L$ of $1_{\pi(A) \cap L}$. Since $x \in \pi(A)$, for infinitely many $m \in \mathbb{N}$ there exists $Z_m \in \alpha_m$ such that $x \in J_{Z_m}$ and $T_{Z_m}(J_{Z_m}) = J_1$. This again together with Lemma 3.6 yields that for each $\varepsilon > 0$ there exist $m \in \mathbb{N}$ and $Z \in \alpha_m$ such that
\[ T_Z(J_Z) = J_1 \quad \text{ and } \quad \lambda_J(J_Z \cap L) \geq (1-\varepsilon) \lambda_J(J_Z).\]
Similar as before, this gives $\lambda_J(J_1 \backslash L) = 0$, so $\lambda_J(J \backslash L) = 0$.
\end{proof}

\subsection{The proof of Theorem~\ref{MAIN}}\label{subsec:3.3}

In the previous paragraphs we collected all the ingredients necessary to prove Theorem~\ref{MAIN}.

\begin{proof}[Proof of Theorem~\ref{MAIN}]

(1) We have constructed a finite $F_Y$-invariant measure  $\mathbb P_C \times \nu$ which is absolutely continuous with respect to $\mathbb P_C\times\lambda_J$. Since $F$ is non-singular with respect to $\mathbb{P} \times \lambda$, we can therefore by Lemma \ref{l:inducedmeasure} extend $\mathbb P_C \times \nu$ to an $F$-invariant measure  $\mathbb P \times \mu$ which is absolutely continuous with respect to $\mathbb P\times\lambda$. Lemma \ref{lemma3.2} immediately implies that $\mu$ is $\sigma$-finite. What is left to show is that $\mathbb P \times \mu$ is the unique such measure (up to multiplication by constants) and that it is ergodic.

A well known result \cite[Theorem 1.5.6]{Aar97} states that a conservative ergodic non-singular transformation $T$ on a probability space $(X,\mathcal B,m)$ admits at most one (up to scalar multiplication) $m$-absolutely continuous   $\sigma$-finite invariant measure. Therefore, it suffices to show that $F$ is conservative and ergodic with respect to $\mathbb P\times\lambda$. We are going to deduce these properties of $F$ from the corresponding properties of the induced transformation $F_Y$.

In the proof of part (2) below we will see that the density of $\frac{d\mu}{d\lambda}$ is bounded away from zero. Hence, $\lambda \ll \mu$. Combining Lemma~\ref{lemma3.2} with Maharam's Recurrence Theorem gives that $F$ is conservative with respect to $\mathbb P \times \mu$ and thus also with respect to $\mathbb P \times \lambda$. Furthermore, from the ergodicity of $F_Y$ with respect to $\mathbb P_C \times \lambda_J$ it follows by Lemma~\ref{lemma3.2} combined with \cite[Proposition 1.5.2(2)]{Aar97} that $F$ is ergodic with respect to $\mathbb P \times \lambda$.

\medskip
(2) For the density $\psi := \frac{d\mu}{d\lambda}$ it holds that $\psi |_J = \varphi$. Since we can take $\kappa$ in the definition of $J$ as large as we want, $\psi$ is locally Lipschitz on $(0,c)$ and $(c,1)$. Moreover, it is a fixed point of the Perron-Frobenius operator from \eqref{eqn3.22} and thus for all $x \in [0,1]$,
\begin{align}\label{q:psipf}
\psi(x) = \mathcal{P}_F^\kappa \psi (x) \geq p_g^\kappa \frac{\varphi(T_g^{-\kappa}x)}{|DT_g^\kappa(T_g^{-\kappa}x)|}.
\end{align}
From Lemma \ref{lemma3.4} we conclude that $\psi$ is bounded from below by some constant $C>0$. It remains to show that $\psi$ is not in $L^q$ for any $q > 1$. To see this, fix a $b \in \Sigma_B$. Since $\psi$ is bounded from below by $C > 0$, we have for all $k \in \mathbb{Z}_{\geq 0}$ and $x \in [0,1]$ that
\begin{align}
\psi(x) = \mathcal{P}_F^{k+1} \psi(x) \geq C \cdot p_g p_b^k \sum_{y \in (T_g T_b^k)^{-1}\{x\}} \frac{1}{|D(T_g T_b^k)(y)|}.
\end{align}
Let $\ell_b,M_b,r_g,M_g,K_g$ be as in (B3) and (G3).
From (B3), (G3) and Lemma~\ref{lemma3.6} we get
\begin{equation}\label{q:q1} \begin{split}
|D(T_gT_b^k)(y)| =\ & |DT_g (T_b^k(y))| \prod_{i=1}^{k} |DT_b (T_b^{k-i}(y))|\\
 \le  \, & M_g |T_b^k(y)-c|^{r_g-1} \prod_{i=0}^{k-1} (M_b|T_b^i(y)-c|^{\ell_b-1})\\ 
\le \, & M_g M_b^k (\tilde M|y-c|)^{\ell_b^k(r_g-1)}  \prod_{i=0}^{k-1} (\tilde M |y-c|)^{\ell_b^i(\ell_b-1)} \\
= \, & K_1 |y-c|^{\ell_b^k r_g-1},
\end{split}\end{equation}
for the positive constant $K_1 = M_g M_b^k \tilde M^{\ell_b^k r_g-1}$. On the other hand, from (G3) we obtain for any $y \in (T_g T_b^k)^{-1}\{x\}$ as in the proof of Lemma~\ref{lemma3.6} that
\[ |x-T_g(c)| = |T_gT_b^k(y)-T_g(c)| \ge  \frac{K_g}{r_g} |T_b^k(y)-c|^{r_g}\]
and then Lemma~\ref{lemma3.6} yields
\begin{equation}\label{q:q2}
|x-T_g(c)| \ge K_2 |y-c|^{\ell_b^k r_g}
\end{equation}
for the positive constant $K_2 = \frac{K_g}{r_g} \tilde K^{\ell_b r_g}$. Now for any $q > 1$ we can choose $k \in \mathbb{Z}_{\geq 0}$ large enough so that $\tau := (1-\ell_b^{-k} r_g^{-1}) q \geq 1$. Combining \eqref{q:psipf}, \eqref{q:q1} and \eqref{q:q2} we obtain
\[ \begin{split}
\psi^q(x) \ge \ & \Big( \frac{C p_g p_b^k}{K_1} \Big)^q \Big( \sum_{y \in (T_g T_b^k)^{-1}\{x\}} |y-c|^{1-\ell_b^k r_g} \Big)^q\\
\ge \ & K_3 |x-T_g(c)|^{-\tau}
\end{split}\]
for a positive constant $K_3$. This gives the result.
\end{proof}

\begin{remark}\label{remark3.1}
The result from Theorem \ref{MAIN} still holds if we allow the critical order $\ell_b$ from (B3) to be equal to $1$ for some $b$, as long as $\ell_{\max} > 1$. To see this, note that in the proof of Theorem \ref{MAIN} condition (B3) only plays a role in proving that $\frac{d\mu_{\mathbf p}}{d\lambda} \not \in L^q$ for any $q > 1$. Here we refer to Lemma \ref{lemma3.6} and the constants $\tilde{K}$ and $\tilde{M}$, which are not well defined if $\ell_{\min} = 1$. In \eqref{q:q1} however, we use the estimates from Lemma~\ref{lemma3.6} only for one arbitrary fixed $b \in \Sigma_B$. By the same reasoning as in the proof of Lemma \ref{lemma3.6} it follows that
\begin{align}
\left( \Big(\frac{K_b}{\ell_b}\Big)^{\frac{1}{\ell_b-1}} |x-c| \right)^{\ell_b^n}  \le  |T_b^n(x)-c| \leq \left( \Big(\frac{M_b}{\ell_b}\Big)^{\frac{1}{\ell_b-1}} |x-c|\right)^{\ell_b^n}.
\end{align}
for any $b \in \Sigma_B$ with $\ell_b>1$. Hence, if there exists at least one $b \in \Sigma_B$ with $\ell_b>1$, then we can replace the bounds obtained from Lemma~\ref{lemma3.6} in \eqref{q:q1} and \eqref{q:q2} by constants $K_1 = M_g M_b^k (\frac{K_b}{\ell_b})^{(\ell_b^k r_g - 1)/(\ell_b-1)}$ and $K_2 = \frac{K_g}{r_g}(\frac{M_b}{\ell_b})^{\ell_b r_g /(\ell_b-1)}$ and obtain the same result. In case $\ell_{\max} = 1$, then most parts from Theorem \ref{MAIN} still remain valid with the exception that then we can only say that $\frac{d\mu_{\mathbf p}}{d\lambda} \not \in L^q$ if $q \geq \frac{r_{\max}}{r_{\max}-1}$. This follows from the above reasoning by taking $k=0$ in the definition of $\tau$ in the proof of Theorem~\ref{MAIN} and by noting that $\tau = (1-r_{\max}^{-1})q \geq 1$ if $q \geq \frac{r_{\max}}{r_{\max}-1}$.
\end{remark}

\section{Estimates on the acs measure}\label{s:in-finite}
In this section we prove Theorem~\ref{MAIN2}. Recall the definition of $\theta$ from Theorem~\ref{MAIN2}:
\[ \theta = \sum_{b \in \Sigma_B} p_b \ell_b.\]

\subsection{The case $\theta \geq 1$}\label{subsec4.1}

To prove one direction of Theorem~\ref{MAIN2}, namely that the unique acs measure $\mu$ from Theorem~\ref{MAIN} is infinite if $\theta \ge 1$, we introduce another induced transformation.

\begin{prop}\label{prop3.2}
Suppose $\theta \geq 1$. Then the unique acs measure $\mu$ from Theorem~\ref{MAIN} is infinite.
\end{prop}

\begin{proof}
Fix a $b \in \Sigma_B$. Recall the definitions of $\tilde{M}$ from Lemma~\ref{lemma3.6} and $\delta$ from in and below the proof of Lemma~\ref{lemma3.6}, and set $\gamma = \min\{\delta,\frac{1}{2}\tilde{M}^{-1}\}$. Let $a \in [c-\gamma,c)$. Then there exists a $\xi \in (a,c)$ such that $T_b(a) > \xi$ and $T_b^2(a) > \xi$. Take $[bb] \times (a, \xi)$ as the inducing domain and let
\begin{align}
\kappa (\omega,x) = \inf\{k \in \mathbb{N}: F^k(\omega,x) \in [bb] \times (a,\xi)\}
\end{align}
be the first return time to $[bb] \times (a,\xi)$ under $F$. If $\mathbb P \times \mu ([bb] \times (a,\xi))=\infty$, there is nothing left to prove. If not, then we compute $\int_{[bb] \times (a, \xi)} \kappa \,  d\mathbb{P} \times \mu$ and use Kac's Formula from Lemma~\ref{l:kac} to prove the result.

\medskip
So, assume that $\mathbb P \times \mu ([bb] \times (a,\xi)) < \infty$. The conditions that $T_b(a) > \xi$ and $T_b^2(a) > \xi$ together with the fact that any bad map has $c$ as a fixed point and is strictly monotone on the intervals $[0,c]$ and $[c,1]$, guarantee that for each $n \in \mathbb{N}$ and $\omega \in \Sigma_B^{\mathbb{N}} \cap [bb]$ we get
\begin{align}\label{eqn3.30}
T_{\omega}^n((a, \xi)) \cap (a, \xi) = \emptyset.
\end{align}
For any $\omega \in [bb]$ and $x \in (a,\xi)$ it follows by \eqref{eqn3.30} and \eqref{q:distancetoc} that $T_\omega^n(x)$ can only return to $(a,\xi)$ after at least one application of a good map. Assume that $\omega \in [bb]$ is of the form
\[ \omega = (b, b, \omega_3, \omega_4, \ldots, \omega_n, g, \omega_{n+2}, \ldots),\]
with $n \ge 2$, $\omega_i \in \Sigma_B$ for $3 \le i \le n$, $g \in \Sigma_g$, and $x \in (a, \xi)$. Then $\kappa (\omega,x) \ge n+1$. Lemma \ref{lemma3.6} yields that
\begin{align}\label{eqn3.31}
|T_{\omega}^n(x)-c| \leq (\tilde M \gamma)^{\ell_{\omega_1} \cdots \ell_{\omega_n}} < 2^{-\ell_{\omega_1} \cdots \ell_{\omega_n}}.
\end{align}
From (G3) and \eqref{eqn3.31} we obtain that 
\begin{align}\label{eqn3.32}
|T_g T_{\omega}^n(x)-T_g(c)| = \left| \int_c^{T_{\omega}^n(x)} DT_g(y) \, dy \right| \le \frac{M_g}{r_g} |T_\omega^n(x)-c|^{r_g} < \frac{M_g}{r_g} \cdot 2^{-\ell_{\omega_1} \cdots \ell_{\omega_n} r_g}.
\end{align}
Set
\begin{align}
\zeta = \sup\{|DT_j(x)|\, :\,  j \in \Sigma,\,  x \in [0,1]\}.
\end{align}
Then $\zeta >1$ by (G4), (B4). Assume $\kappa(\omega,x) = m+n$ for some $m \ge 1$. Then $T_\omega^{m+n}(x) \in (a, \xi)$ so that by (G1),
\begin{align}
|T_{\omega}^{m+n}(x)-T_g(c)| \geq \min\{a,1-\xi\}.
\end{align}
Because of \eqref{eqn3.32} this implies
\begin{align}
\zeta^{m-1} \frac{M_g}{r_g} \cdot 2^{-\ell_{\omega_1} \cdots \ell_{\omega_n} r_g} \geq \min\{a,1-\xi\}.
\end{align}
Solving for $m$ yields
\begin{align}
m \geq K_1 + K_2 \ell_{\omega_1} \cdots \ell_{\omega_n}
\end{align}
for constants $K_1 = \big(1+\log \big( \frac{\min\{ a, 1-\xi \} r_g}{M_g} \big)\big) / \log \zeta \in \mathbb R$ and $K  _2 = \log(2^{r_g})/\log \zeta >0$. Note that $K_1,K_2$ are independent of $\omega, x, m$ and $n$.

\medskip

We obtain that for any $g \in \Sigma_G$,
\[ \begin{split}
\int_{[bb] \times (a, \xi)} \kappa \,  d\mathbb{P} \times \mu \ge \, & \sum_{n \in \mathbb{N}_{\ge 2}} \sum_{\omega_3,\ldots,\omega_n \in \Sigma_B} \mathbb{P}([bb \omega_3 \cdots \omega_n g]) \mu((a, \xi)) \Big(n+K_1 + K_2 \ell_b^2 \prod_{i=3}^n \ell_{\omega_i}\Big).
\end{split}\]
Since
\[  \sum_{n \in \mathbb{N}_{\ge 2}} \sum_{\omega_3,\ldots,\omega_n \in \Sigma_B}  \mathbb{P}([\omega_3 \cdots \omega_n]) \prod_{i=3}^n \ell_{\omega_i} = 1+\sum_{n \in \mathbb N} \theta^n = \infty,\]
we get $\int_{[bb] \times (a, \xi)} \kappa \,  d\mathbb{P} \times \mu = \infty$ and from Lemma~\ref{l:kac} we now conclude that $\mu$ is infinite.
\end{proof}

\subsection{The case $\theta < 1$} \label{sec3.3}
For the other direction of Theorem~\ref{MAIN2}, assume $\theta < 1$. We first obtain a stationary probability measure $\tilde{\mu}$ for $F$ as in \eqref{q:skewproduct} using a standard Krylov-Bogolyubov type argument. For this, let $\mathcal{M}$ denote the set of all finite Borel measures on $[0,1]$, and define the operator $\mathcal{P}: \mathcal{M} \rightarrow \mathcal{M}$ by
\begin{align}
\mathcal{P} \nu = \sum_{j \in \Sigma} p_j \nu \circ T_j^{-1}, \qquad \nu \in \mathcal{M},
\end{align}
where $\nu \circ T_j^{-1}$ denotes the pushforward measure of $\nu$ under $T_j$. Then $\mathcal{P}$ is a \emph{Markov-Feller} operator (see e.g.~\cite{Las04}) with dual operator $U$ on the space $B([0,1])$ of all bounded Borel measurable functions given by\footnote{By definition of a Markov-Feller operator, the space of bounded \emph{continuous} functions is required to be invariant under the dual operator $U$. If there is a $g \in \Sigma_G$ for which $T_g$ is discontinuous (namely at $c$), we then first identify $[0,1]$ with the unit circle $S^1$ so that $T_g$ can be viewed as a continuous map on $S^1$. With the same identification any acs measure on $S^1$ then gives an acs measure on $[0,1]$.} $U f = \sum_{j \in \Sigma} p_j f \circ T_j$ for $f \in B([0,1])$. As before, let $\lambda$ denote the Lebesgue measure on $[0,1]$, and set $\lambda_n = \mathcal{P}^n \lambda$ for each $n \geq 0$. Furthermore, for each $n \in \mathbb{N}$ define the Ces\'aro mean $\mu_n = \frac{1}{n} \sum_{k=0}^{n-1} \lambda_k$. Since the space of probability measures on $[0,1]$ equipped with the weak topology is sequentially compact, 
 there exists a subsequence $(\mu_{n_k})_{k \in \mathbb{N}}$ of $(\mu_n)_{n \in \mathbb{N}}$ that converges weakly to a probability measure $\tilde{\mu}$ on $[0,1]$. Using that a Markov-Feller operator is weakly continuous, it then follows from a standard argument that 
$\mathcal{P} \tilde{\mu} = \tilde{\mu}$, that is, $\tilde{\mu}$ is a stationary probability measure for $F$. The next theorem will lead to the estimate \eqref{eq:2.10} from Theorem~\ref{MAIN2}. For any $\mathbf b = b_1 \cdots b_k \in \Sigma_B^k$, $k \ge 0$, recall that we abbreviate $p_{\mathbf b} = \prod_{i=1}^kp_{b_i} $ and also let $\ell_{\mathbf b} = \prod_{i=1}^k \ell_{b_i}$ where we use $p_{\mathbf b}=1=\ell_{\mathbf b}$ in case $k=0$.

\begin{theorem}\label{thrm3.1}
There exists a constant $C > 0$ such that for all $n \in \mathbb{N}$ and all Borel sets $A \subseteq [0,1]$ we have
\begin{align}\label{eqn3.40}
\lambda_n(A) \leq C \cdot \sum_{g \in \Sigma_G} p_g \sum_{k=0}^{\infty} \sum_{\mathbf{b} \in \Sigma_B^k} p_{\mathbf{b}} \ell_{\mathbf{b}} \cdot \lambda(A)^{\ell_{\mathbf{b}}^{-1} r_g^{-1}}.
\end{align}
\end{theorem}

Before we prove this theorem, we first show how it gives Theorem~\ref{MAIN2}.

\begin{proof}[Proof of Theorem~\ref{MAIN2}]
The first part of the statement follows from Proposition~\ref{prop3.2}. For the second part, assume that $\theta < 1$ and that Theorem~\ref{thrm3.1} holds. Let $A \subseteq [0,1]$. Using the regularity of $\lambda$, for any $ \delta >0$ there exists an open set $G \subseteq [0,1]$ such that $A \subseteq G$ and $\lambda(G) \leq \lambda(A) + \delta$. Using that $(\mu_{n_k})_{k \in \mathbb{N}}$ converges weakly to $\tilde{\mu}$, we obtain from the Portmanteau Theorem together with Theorem \ref{thrm3.1} that
\begin{align}
\tilde{\mu}(A) & \leq \tilde{\mu}(G) \leq \liminf_k \mu_{n_k}(G) \notag \\
& \leq C \cdot \sum_{g \in \Sigma_G} p_g \sum_{k=0}^{\infty} \sum_{\mathbf{b} \in \Sigma_B^k} p_{\mathbf{b}} \ell_{\mathbf{b}} \cdot (\lambda(A)+\delta)^{\ell_{\mathbf{b}}^{-1} r_g^{-1}}. \label{eqn3.41}
\end{align}

Since $\theta < 1$, the sum is bounded and with the Dominated Convergence Theorem we can take the limit as $\delta \to 0$ to obtain
\begin{align}\label{eqn3.42}
\tilde{\mu}(A) \leq C \cdot \sum_{g \in \Sigma_G} p_g \sum_{k=0}^{\infty} \sum_{\mathbf{b} \in \Sigma_B^k} p_{\mathbf{b}} \ell_{\mathbf{b}} \cdot \lambda(A)^{\ell_{\mathbf{b}}^{-1} r_g^{-1}}.
\end{align}
This proves that $\tilde{\mu}$ is absolutely continuous with respect to the Lebesgue measure on $[0,1]$. It follows that the probability measure $\tilde{\mu}$ is equal to the unique acs measure $\mu_{\mathbf p}$ from Theorem \ref{MAIN}. The estimate \eqref{eq:2.10} follows directly from \eqref{eqn3.42}.
\end{proof}

It remains to give the proof of Theorem \ref{thrm3.1}. We shall do this in a number of steps. 

\begin{prop}\label{prop0.6}
There exists a constant $K_1 > 0$ such that for all $n \in \mathbb{N}$, all $\mathbf u \in \Sigma^n$ and all Borel sets $A \subseteq [0,1]$ with $0 <  3\lambda(A) < \frac{1}{2} \min\{c,1-c\} $ 
we have 
\[  \lambda(T_{\mathbf u}^{-1}A) \leq K_1 \big( \lambda(T_{\mathbf u}^{-1}[0,3\eta)) + \lambda(T_{\mathbf u}^{-1}(c-3\eta,c+3\eta)) + \lambda(T_{\mathbf u}^{-1}(1-3\eta,1])\big),\]
where $\eta = \lambda(A)$.
\end{prop}

\begin{proof}
Let $n \in \mathbb{N}$, $\mathbf u \in \Sigma^n$ and a Borel set $A \subseteq [0,1]$ with $0 <  3\lambda(A) < \frac{1}{2} \min\{c,1-c\} < 1$ be given and write $\eta = \lambda(A)$. The map $T_{\mathbf u}$ has non-positive Schwarzian derivative on any of its intervals of monotonicity (see \eqref{eqn5}) and the image of any such interval is $[0,c], [c,1]$ or $[0,1]$. Set $A_1 = (\eta,c-\eta)$ and $A_2= (2\eta,c-2\eta)$. Let $I$ be a connected component of $T_{\mathbf u}^{-1} A_1$, and set $f = T_{\mathbf u}|_{I}$ and $I^* = f^{-1} A_2$. The Minimum Principle yields
\begin{align}\label{eq19}
|Df(x)| \geq \min_{z \in \partial I^*} |Df(z)|, \qquad \text{ for all $x \in I^*$.} 
\end{align}
Suppose the minimal value is attained at $f^{-1}(2\eta)$ and set $A_3 = (2\eta,3\eta)$ and  $J = f^{-1} A_3$. By the condition on the size of $A$ it follows from the Koebe Principle that
\begin{align}\label{eq20}
K^{(\eta)} |Df(f^{-1}(2\eta))| \geq |Df(x)|, \qquad \text{ for all $x \in J$.}
\end{align}
Combining \eqref{eq19} and \eqref{eq20} gives
\[ \begin{split}
\lambda (f^{-1}(A \cap A_2)) =\ & \int_{A \cap A_2} \frac{1}{|Df(f^{-1} y)|} \, d\lambda(y) \leq \lambda(A) \cdot \frac{1}{|Df(f^{-1}(2\eta))|} \\
\leq \ & K^{(\eta)} \int_{A_3} \frac{1}{|Df(f^{-1}y)|} \, d\lambda(y) = K^{(\eta)} \lambda(f^{-1}(A_3)).
\end{split}\]
We conclude that
\begin{align}\label{eq21}
\lambda \big(T_{\mathbf u}^{ -1}\big(A \cap (2\eta,c-2\eta)\big)\big) \leq K^{(\eta)} \lambda \big( T_{\mathbf u}^{-1}(2\eta,3\eta)\big).
\end{align}
In case $\min_{z \in \partial I^*} |Df(z)| = f^{-1}(c-2\eta)$, a similar reasoning yields
\begin{align}
\lambda \big( T_{\mathbf u}^{ -1}\big(A \cap (2\eta,c-2\eta)\big)\big)  \leq K^{(\eta)} \lambda \big( T_{\mathbf u}^{-1}(c-3\eta,c-2\eta)\big).
\end{align}
Furthermore, a similar reasoning can be done for the interval $[c,1]$ to conclude that
\[ \lambda \big( T_{\mathbf u}^{ -1}\big(A \cap (c+2\eta,1-2\eta)\big)\big) \leq  K^{(\eta)} \Big( \lambda \big( T_{\mathbf u}^{-1}(c+2\eta,c+3\eta) \big) + \lambda \big( T_{\mathbf u}^{-1}(1-3\eta,1-2\eta) \big) \Big). \]
Hence, setting $K_1 = \max\{K^{(\eta)},1\}$ gives the desired result. 
\end{proof}

Proposition~\ref{prop0.6} shows that to get the desired estimate from Theorem \ref{thrm3.1} it suffices to consider small intervals on the left and right of $[0,1]$ and around $c$, i.e., sets of the form
\[ I_c(\varepsilon):=(c-\varepsilon,c+\varepsilon) \quad  \text{and} \quad I_0(\varepsilon):=[0, \varepsilon ) \cup (1-\varepsilon,1]\]
for $\varepsilon >0$. We first focus on estimating the measure of the intervals $I_c(\varepsilon)$.
\begin{lemma}\label{lemma0.8}
There exists a constant $K_2 \geq 1$ such that for all $n \in \mathbb{N}$, $\mathbf{u} \in \Sigma^{n-1} \times \Sigma_G$ and all $\varepsilon > 0$ we have
\begin{align}
\lambda(T_{\mathbf u}^{-1} I_c(\varepsilon))\leq K_2 \varepsilon.
\end{align}
\end{lemma}

\begin{proof} Let $n \in \mathbb{N}$ and $\mathbf u \in \Sigma^{n-1} \times \Sigma_G$. Let $\varepsilon > 0$. Suppose that $\varepsilon \ge \frac14 \min\{c,1-c\}$. Then
\begin{align}\label{eq3.53}
\lambda(T_{\mathbf u}^{-1}I_c(\varepsilon)) \leq 1 \leq \frac{4\varepsilon }{\min\{c,1-c\}}.
\end{align}
Now suppose $\varepsilon < \frac14 \min\{c,1-c\}$. Again the map $T_{\mathbf u}$ has non-positive Schwarzian derivative on the interior of any of its intervals of monotonicity and since $ u_n \in \Sigma_G$ the image of any such interval is $[0,1]$. 
Use $\mathcal I$ to denote the collection of connected components of $T_{\mathbf u}^{-1}I_c(\varepsilon)$. Let $A \in \mathcal I$ and write $J = J_A$ and $I = I_A$ for the intervals that satisfy $A \subseteq J$, $A \subseteq I$ and 
\[ \begin{split}
T_{\mathbf u}(J) =\ & \Big[c-\frac12 \min\{c,1-c\},c+\frac12  \min\{c,1-c\}\Big],\\
T_{\mathbf u}(I) =\ & \Big[c-\frac34  \min\{c,1-c\},c+\frac34  \min\{c,1-c\}\Big].
\end{split}\]
Also, write $f = T_{\mathbf u}|_I$. Since $f$ has non-positive Schwarzian derivative, it follows from 
 \eqref{q:intervals} that
\begin{align}
\frac{\lambda(A)}{\lambda(J)} \leq K^{(\frac14)} \frac{\lambda(f(A))}{\lambda(f(J))} = K^{(\frac14)} \frac{2\varepsilon}{\min\{c,1-c\}}.
\end{align}
We conclude that
\begin{align}\label{eq3.57a}
\lambda(T_{\mathbf u}^{-1} I_c(\varepsilon)) = \sum_{A \in \mathcal I} \lambda(A) \leq K^{(\frac14)} \frac{2\varepsilon}{\min\{c,1-c\}} \sum_{A \in \mathcal I}\lambda(J_A) \leq K^{(\frac14)} \frac{2\varepsilon}{\min\{c,1-c\}}.
\end{align}
Defining $K_2 = \frac{2\max\{2,K^{(\frac14)}\}}{\min\{c,1-c\}} $, the desired result now follows from \eqref{eq3.53} and \eqref{eq3.57a}.
\end{proof}

To find $\lambda_n\big(I_c(\varepsilon)\big)$, first note that from Lemma \ref{lemma3.6} it follows that for all $\varepsilon >0$, $n \in \mathbb N$, $\mathbf u \in \Sigma_B^n$,
\begin{equation}\label{eq3.57}
T_{\mathbf u}^{-1} \big( I_c(\varepsilon) \big) \subseteq I_c \big( \tilde K^{-1}\varepsilon^{\ell_{\mathbf u}^{-1}} \big).
\end{equation}
By splitting $\Sigma^n$ according to the final block of bad indices, we can then write using \eqref{eq3.57} and Lemma~\ref{lemma0.8} that
\[ \begin{split}
\lambda_n\big(I_c(\varepsilon)\big)  =\ & \sum_{k=0}^{n-1} \sum_{\mathbf{v} \in \Sigma^{n-k-1}} \sum_{g \in \Sigma_G} \sum_{\mathbf{b} \in \Sigma_B^k} p_{\mathbf{v} g \mathbf{b}} \lambda\big(T^{-1}_{\mathbf{v}g\mathbf{b}} I_c(\varepsilon)\big) + \sum_{\mathbf{b} \in \Sigma_B^n} p_{\mathbf{b}} \lambda\big(T_{\mathbf{b}}^{-1} I_c(\varepsilon)\big)\\
\le \ & \sum_{k=0}^{n-1} \sum_{\mathbf{v} \in \Sigma^{n-k-1}} \sum_{g \in \Sigma_G} \sum_{\mathbf{b} \in \Sigma_B^k} p_{\mathbf{v} g \mathbf{b}} \lambda \big( T_{\mathbf v g}^{-1} I_c (\tilde K^{-1} \varepsilon^{\ell_{\mathbf b}^{-1}}) \big) + \sum_{\mathbf{b} \in \Sigma_B^n} p_{\mathbf{b}} \lambda\big( I_c( \tilde K^{-1} \varepsilon^{\ell_{\mathbf b}^{-1}})\big)\\
\le \ & \sum_{k=0}^{n-1} \sum_{g \in \Sigma_G} \sum_{\mathbf{b} \in \Sigma_B^k} p_g p_{\mathbf{b}} K_2 \tilde K^{-1} \varepsilon^{\ell_{\mathbf b}^{-1}} + \sum_{\mathbf{b} \in \Sigma_B^n} p_{\mathbf b} 2\tilde K^{-1} \varepsilon^{\ell_{\mathbf{b}}^{-1}}.
\end{split}\]
Taking $K_3 = \max \big\{ K_2, 2 \big( \sum_{g \in \Sigma_G} p_g \big)^{-1} \big\} \cdot \tilde K^{-1} \ge 1$ then gives
\begin{equation}\label{q:measureIm}
\lambda_n \big(I_c(\varepsilon)\big)  \le K_3 \sum_{g \in \Sigma_G} \sum_{k=0}^n \sum_{\mathbf b \in \Sigma_B^k} p_g p_{\mathbf b} \varepsilon^{\ell_{\mathbf b}^{-1}}.
\end{equation}

\medskip
We now focus on $I_0(\varepsilon) = [0, \varepsilon) \cup (1-\varepsilon,1]$. Fix an $ 0 < \varepsilon_0 < \frac12 \min\{ c, 1-c\}$ and a $t > 1$ that satisfy
\begin{align}\label{eq31}
|DT_j(x)| > t, \qquad \text{ for all $x \in I_0(\varepsilon_0)$ and each $j \in \Sigma$}.
\end{align}
Such $\varepsilon_0$ and $t$ exist because of (G4) and (B4). From (G3) it follows that for each $0 < \varepsilon < \varepsilon_0$ and $g \in \Sigma_G$,
\[ |T_g(x)-T_g(c)| = \left| \int_c^x DT_g (y) dy \right| \ge \frac{K_g}{r_g} \cdot |x-c|^{r_g}.\]
Set $K_4 = \max\{(K_g^{-1} r_g)^{r_g^{-1}}: g \in \Sigma_G\} \geq 1$. Then (G1) implies that
\begin{equation}\label{eq3.60}
T_g^{-1} I_0(\varepsilon) \subseteq  I_0(\varepsilon t^{-1}) \cup I_c (K_4 \varepsilon^{r_g^{-1}} ).
\end{equation}
Furthermore, from (B1) it follows that for each $\varepsilon \in (0,\varepsilon_0)$ and $b \in \Sigma_B$,
\begin{align}\label{eq3.61}
T_b^{-1} I_0(\varepsilon) \subseteq I_0(\varepsilon t^{-1}).
\end{align}
Write each $\mathbf{u} \in \Sigma^n$ as
\begin{align}\label{eq:4.30}
\mathbf{u} = \mathbf{b}_1 \mathbf{g}_1 \cdots \mathbf{b}_{\tilde{s}} \mathbf{g}_{\tilde{s}}
\end{align}
for some $\tilde{s} \in \{1,\ldots,n\}$, where for each $i$ we have $\mathbf{b}_i = b_{i,1} \cdots b_{i,k_i}\in \Sigma_B^{k_i}$ and $\mathbf{g}_i = g_{i,1} \cdots g_{i,m_i} \in \Sigma_G^{m_i}$ for some $k_1,m_{\tilde{s}} \in \mathbb{Z}_{\geq 0}$ and $k_2,\ldots,k_{\tilde{s}}, m_1,\ldots,m_{\tilde{s}-1} \in \mathbb{N}$. Define 
\[ s= \begin{cases}
\tilde s, & \text{ if } m_{\tilde{s}} \geq 1,\\
\tilde{s} -1, & \text{ if } m_{\tilde{s}} = 0 .
\end{cases}\]
Moreover, we introduce notation to indicate the length of the tails of the block $\mathbf u$:
\[ \begin{array}{ll}
d_i = |\mathbf{b}_i \mathbf{g}_i \cdots \mathbf{b}_{\tilde{s}} \mathbf{g}_{\tilde{s}}|, & i \in \{1,\ldots,\tilde{s}\}, \\
q_{i,j} = |g_{i,j+1} \cdots g_{i,m_i} \mathbf{b}_{i+1} \mathbf{g}_{i+1} \cdots \mathbf{b}_{\tilde{s}} \mathbf{g}_{\tilde{s}}|,  & i \in \{1,\ldots,\tilde{s}\}, \, j \in \{0,\ldots,m_i\}.
\end{array}\]
If necessary to avoid confusion, we write $s(\mathbf{u})$, $k_i(\mathbf{u})$, etcetera to emphasize the dependence on $\mathbf{u}$.

\begin{lemma}\label{l:c5}
There exists a constant $K_5 > 0$ such that for each $ 0 < \varepsilon < \varepsilon_0$, $n \in \mathbb{N}$ and $\mathbf{u} = \mathbf{b}_1 \mathbf{g}_1 \cdots \mathbf{b}_{\tilde{s}} \mathbf{g}_{\tilde{s}} \in \Sigma^n$,
\[\begin{split}
T_{\mathbf{u}}^{-1} I_0(\varepsilon) \subseteq \ & I_0(\varepsilon t^{-d_1}) \cup \bigcup_{i=1}^{s} T^{-1}_{\mathbf{b}_1 \mathbf{g}_1 \cdots \mathbf{b}_{i-1} \mathbf{g}_{i-1}} I_c\Big(K_5(\varepsilon t^{-q_{i,1}})^{\ell_{\mathbf{b}_i}^{-1} r^{-1}_{g_{i,1}}}\Big) \notag \\
& \cup \bigcup_{i=1}^s \bigcup_{j=2}^{m_i} T_{\mathbf{b}_1 \mathbf{g}_1 \cdots \mathbf{b}_{i-1} \mathbf{g}_{i-1}\mathbf{b}_i g_{i,1} \cdots g_{i,j-1}}^{-1}I_c(K_5 (\varepsilon t^{-q_{i,j}})^{r_{g_{i,j}}^{-1}}).
\end{split}\]
\end{lemma}

\begin{proof}
We prove the statement by an induction argument for $\tilde{s}$. Let $\mathbf{u}$ be a word with symbols in $\Sigma$, and write $\mathbf{u} = \mathbf{b}_1 \mathbf{g}_1 \cdots \mathbf{b}_{\tilde{s}} \mathbf{g}_{\tilde{s}}$ for its decomposition as in \eqref{eq:4.30}. First suppose that $\tilde{s} = 1$. If $m_1 = 0$, then the statement immediately follows from repeated application of \eqref{eq3.61}. If $m_1 \geq 1$, then repeated application of \eqref{eq3.60} gives
\begin{equation}
T_{\mathbf{g}_1}^{-1} I_0(\varepsilon)  \subseteq I_0 (\varepsilon t^{-q_{1,0}}) \cup I_c\Big(K_4(\varepsilon t^{-q_{1,1}})^{r^{-1}_{g_{1,1}}}\Big) \cup \bigcup_{j=2}^{m_1} T_{g_{1,1} \cdots g_{1,j-1}}^{-1} I_c\Big(K_4(\varepsilon t^{-q_{1,j}})^{r^{-1}_{g_{1,j}}}\Big).
\end{equation}
By setting $K_5 = \tilde K^{-1} K_4$, applying \eqref{eq3.57} and \eqref{eq3.61} then yields
\[ T_{\mathbf{b}_1 \mathbf{g}_1}^{-1} I_0 (\varepsilon) \subseteq I_0 (\varepsilon t^{-d_1}) \cup I_c\Big(K_5(\varepsilon t^{-q_{1,1}})^{\ell_{\mathbf{b}_1}^{-1} r^{-1}_{g_{1,1}}}\Big) \cup \bigcup_{j=2}^{m_1} T_{\mathbf{b}_1 g_{1,1} \cdots g_{1,j-1}}^{-1} I_c\Big(K_5(\varepsilon t^{-q_{1,j}})^{r^{-1}_{g_{1,j}}}\Big).\]
Note that this is true for the case that $k_1 = 0$ as well. This proves the statement if $\tilde{s} = 1$. Now suppose $\tilde{s}(\mathbf{u}) > 1$ and suppose that the statement holds for all words $\mathbf{v}$ with $\tilde{s}(\mathbf{v}) = \tilde{s}(\mathbf{u})-1$. In particular, the statement then holds for the word $\mathbf{b}_2 \mathbf{g}_2 \cdots \mathbf{b}_{\tilde{s}} \mathbf{g}_{\tilde{s}}$. Note that $m_1 \geq 1$. Again, by repeated application of \eqref{eq3.60} it follows that
\begin{equation}\label{eq37}
T_{\mathbf{g}_1}^{-1} I_0(\varepsilon t^{-d_2})  \subseteq I_0 (\varepsilon t^{-q_{1,0}}) \cup I_c\Big(K_4(\varepsilon t^{-q_{1,1}})^{r^{-1}_{g_{1,1}}}\Big) \cup \bigcup_{j=2}^{m_1} T_{g_{1,1} \cdots g_{1,j-1}}^{-1} I_c\Big(K_4(\varepsilon t^{-q_{1,j}})^{r^{-1}_{g_{1,j}}}\Big).
\end{equation}
Furthermore, applying \eqref{eq3.57} and \eqref{eq3.61} then yields
\[ T_{\mathbf{b}_1 \mathbf{g}_1}^{-1} I_0 (\varepsilon t^{-d_2}) \subseteq I_0 (\varepsilon t^{-d_1}) \cup I_c\Big(K_5(\varepsilon t^{-q_{1,1}})^{\ell_{\mathbf{b}_1}^{-1} r^{-1}_{g_{1,1}}}\Big) \cup \bigcup_{j=2}^{m_1} T_{\mathbf{b}_1 g_{1,1} \cdots g_{1,j-1}}^{-1} I_c\Big(K_5(\varepsilon t^{-q_{1,j}})^{r^{-1}_{g_{1,j}}}\Big).\]
This together with the statement being true for the word $\mathbf{b}_2 \mathbf{g}_2 \cdots \mathbf{b}_{\tilde{s}} \mathbf{g}_{\tilde{s}}$ yields the statement for $\mathbf{u}$.
\end{proof}

Combining Lemma~\ref{lemma0.8} and Lemma~\ref{l:c5} gives
\[\lambda(T_{\mathbf u}^{-1}I_0(\varepsilon)) \le 2\varepsilon t^{-d_1} + \sum_{i=1}^s K_2 K_5 (\varepsilon t^{-q_{i,1}})^{\ell_{\mathbf{b}_i}^{-1} r^{-1}_{g_{i,1}}} + \sum_{i=1}^s \sum_{j=2}^{m_i} K_2 K_5 (\varepsilon t^{-q_{i,j}})^{r_{g_{i,j}}^{-1}}.\]
Let $r_{\max} = \max \{ r_g \, : \, g \in \Sigma_G\}$ and set $\alpha := t^{1/r_{\max}}>1$. Then
\[ \sum_{i=1}^s \sum_{j=2}^{m_i} \alpha^{-q_{i,j}} \leq \sum_{\ell=0}^{\infty} \alpha^{-\ell} =  \frac{1}{1-1/\alpha}, \]
so that
\begin{equation}\label{q:c2c5} \begin{split}
\lambda(T_{\mathbf u}^{-1}I_0(\varepsilon)) \le \ & 2 \varepsilon^{1/r_{\max}} + K_2 K_5 \sum_{i=1}^s \sum_{j=2}^{m_i} \varepsilon^{1/r_{\max}} \alpha^{-q_{i,j}} + \sum_{i=1}^s K_2 K_5 (\varepsilon t^{-q_{i,1}})^{\ell_{\mathbf{b}_i}^{-1} r^{-1}_{g_{i,1}}} \\
\le \ &  \bigg( 2 + \frac{K_2K_5}{1-1/\alpha} \bigg) \varepsilon^{1/r_{\max}} + K_2 K_5\sum_{i=1}^s (\varepsilon t^{-q_{i,1}})^{\ell_{\mathbf{b}_i}^{-1} r^{-1}_{g_{i,1}}}.
\end{split}
\end{equation}

\begin{prop}\label{prop15}
There exists a constant $K_6 > 0$ such that for each $\varepsilon \in (0,\varepsilon_0)$ and $n \in \mathbb{N}$,
\begin{align*}
\lambda_n (I_0(\varepsilon)) \leq K_6 \sum_{g \in \Sigma_G} p_g  \sum_{k=0}^{n-1} \sum_{\mathbf{b} \in \Sigma_B^k} p_{\mathbf{b}} \ell_{\mathbf{b}} \cdot \varepsilon^{\ell_{\mathbf{b}}^{-1}r_g^{-1}}.
\end{align*}
\end{prop}

\begin{proof}
Let $n \in \mathbb{N}$. Then with \eqref{q:c2c5} we obtain
\begin{align}
\lambda_n (I_0(\varepsilon)) &= \sum_{\mathbf{u} \in \Sigma^n} p_{\mathbf{u}} \lambda\big(T_{\mathbf{u}}^{-1} (I_0(\varepsilon))\big) \notag \\
& \leq  \bigg( 2 + \frac{K_2K_5}{1-1/\alpha} \bigg) \varepsilon^{1/r_{\max}} + K_2 K_5 \sum_{\mathbf{u} \in \Sigma^n} p_{\mathbf{u}} \sum_{i=1}^{s(\mathbf{u})} \big(\varepsilon t^{-q_{i,1}(\mathbf{u})}\big)^{\ell^{-1}_{\mathbf{b}_i(\mathbf{u})}r^{-1}_{g_{i,1}(\mathbf{u})}} \notag \\
& = \bigg( 2 + \frac{K_2K_5}{1-1/\alpha} \bigg) \varepsilon^{1/r_{\max}} + K_2 K_5  \sum_{i=1}^{\tau} \sum_{\mathbf{u} \in \Sigma^n} 1_{\{1,\ldots,s(\mathbf{u})\}}(i) p_{\mathbf{u}} \big(\varepsilon t^{-q_{i,1}(\mathbf{u})}\big)^{\ell^{-1}_{\mathbf{b}_i(\mathbf{u})}r^{-1}_{g_{i,1}(\mathbf{u})}}, \label{eqnr1}
\end{align}
where we defined $\tau = \floor{\frac{n+1}{2}}$ which is the largest value $s(\mathbf u)$ can take.
Let us consider the second term in \eqref{eqnr1}. First of all, note that a word $\mathbf{u} \in \Sigma^n$ satisfies $s(\mathbf{u}) \geq 1$ if and only if $m_1(\mathbf{u}) \geq 1$. Therefore,
\begin{align*}
\{\mathbf{u} \in \Sigma^n: s(\mathbf{u}) \geq 1\} &= \bigcup_{k=0}^{n-1} \Sigma_B^k \times \Sigma_G \times \Sigma^{n-k-1}.
\end{align*}
Hence, defining the function $\chi$ on $\{0,\ldots,n-1\}^2$ by
\begin{align}\label{defchi}
\chi(k,q) = \sum_{\mathbf{b} \in \Sigma_B^k} \sum_{g \in \Sigma_G} p_{\mathbf{b}} p_g \big(\varepsilon t^{-q}\big)^{\ell_{\mathbf{b}}^{-1} r_g^{-1}}, \qquad (k,q) \in \{0,\ldots,n-1\}^2.
\end{align}
we can rewrite and bound the term with $i=1$ in \eqref{eqnr1} as follows:
\begin{align}
 \sum_{\mathbf{u} \in \Sigma^n} 1_{\{1,\ldots,s(\mathbf{u})\}}(1) p_{\mathbf{u}} \big(\varepsilon t^{-q_{1,1}(\mathbf{u})}\big)^{\ell^{-1}_{\mathbf{b}_i(\mathbf{u})}r^{-1}_{g_{1,1}(\mathbf{u})}} &= \sum_{k=0}^{n-1} \sum_{\mathbf{v} \in \Sigma^{n-k-1}} p_{\mathbf{v}} \chi(k,n-k-1) \notag \\
&\leq \varepsilon^{1/r_{\max}} + \sum_{k=1}^{n-1} \chi(k,n-k-1). \label{eqnr2}
\end{align}
Secondly, note that for each $i \in \{2,\ldots,\tau\}$ a word $\mathbf u \in \Sigma^n$ satisfies $s(\mathbf{u}) \geq i$ if and only if $m_{i-1}(\mathbf{u}),k_i(\mathbf{u}),m_i(\mathbf{u}) \geq 1$. For each $k \in \{1,\ldots,n-1\}$ and $q \in \{0,\ldots,n-k-2\}$ and $i \in \{2,\ldots,\tau\}$ we define
\begin{align}
A_{i,k,q} = \{ \mathbf{v} \in \Sigma^{n-k-q-1}: \tilde{s}(\mathbf{v}) = i-1, v_{n-k-q-1} \in \Sigma_G\}.
\end{align}
The set $A_{i,k,q}$ contains all words of length $n-k-q-1$ that can precede the word $\mathbf b_i \mathbf{g}_i \cdots \mathbf b_{\tilde{s}} \mathbf g_{\tilde{s}}$ with $|\mathbf{b}_i| = k$ and $|g_{i,2} \cdots g_{i,m_i} \mathbf{b}_{i+1} \mathbf g_{i+1} \cdots \mathbf b_{\tilde{s}} \mathbf{g}_{\tilde{s}}| = q$. So
\begin{align*}
\{\mathbf{u} \in \Sigma^n: s(\mathbf{u}) \geq i\} &= \bigcup_{k=1}^{n-1} \bigcup_{q=0}^{n-k-2} A_{i,k,q} \times \Sigma_B^k \times \Sigma_G \times \Sigma^q, \qquad i \in \{2, \ldots, \tau\}.
\end{align*}
Hence, using \eqref{defchi} we can rewrite and bound the sum in \eqref{eqnr1} that runs from $i=2$ to $\tau$ as follows:
\begin{align}
\sum_{i=2}^{\tau} \sum_{\mathbf{u} \in \Sigma^n} 1_{\{1,\ldots,s(\mathbf{u})\}}(i) p_{\mathbf{u}} \big(\varepsilon t^{-q_{i,1}(\mathbf{u})}\big)^{\ell^{-1}_{\mathbf{b}_i(\mathbf{u})}r^{-1}_{g_{i,1}(\mathbf{u})}} =\ & \sum_{i=2}^{\tau} \sum_{k=1}^{n-1} \sum_{q=0}^{n-k-2} \sum_{\mathbf{v}_1 \in A_{i,k,q}} \sum_{\mathbf{v}_2 \in \Sigma^{q-1}} p_{\mathbf{v}_1} p_{\mathbf{v}_2} \chi(k,q) \notag \\
=\ & \sum_{k=1}^{n-1}\sum_{q=0}^{n-k-2} \chi(k,q) \sum_{i=2}^{\tau} \sum_{\mathbf{v}_1 \in A_{i,k,q}} \sum_{\mathbf{v}_2 \in \Sigma^{q-1}} p_{\mathbf{v}_1} p_{\mathbf{v}_2} \notag \\
\le & \sum_{k=1}^{n-1} \sum_{q=0}^{n-k-2} \chi(k,q). \label{eqnr3}
\end{align}
Here the last step follows from the fact that
\begin{align}
\sum_{i=2}^{\tau} \sum_{\mathbf{v}_1 \in A_{i,k,q}} p_{\mathbf{v}_1} \leq \sum_{\mathbf{v} \in \Sigma^{n-k-q-2}} \sum_{g \in \Sigma_G} p_{\mathbf v} p_g \leq 1.
\end{align}
Combining \eqref{eqnr2} and \eqref{eqnr3} gives
\begin{align}\label{eqnr4}
\sum_{i=1}^{\tau} \sum_{\mathbf{u} \in \Sigma^n} 1_{\{1,\ldots,s(\mathbf{u})\}}(i) p_{\mathbf{u}} \big(\varepsilon t^{-q_{i,1}(\mathbf{u})}\big)^{\ell^{-1}_{\mathbf{b}_i(\mathbf{u})}r^{-1}_{g_{i,1}(\mathbf{u})}} \leq \varepsilon^{1/r_{\max}} + \sum_{k=1}^{n-1} \sum_{q=0}^{n-k-1} \chi(k,q).
\end{align}
Furthermore, for each $\mathbf{b} \in \Sigma_B^k$ and $g \in \Sigma_G$ we have again by setting $r_{\max} = \max\{r_j: j \in \Sigma_G\}$ and $\alpha = t^{1/r_{\max}}$ that
\begin{align}\label{eqnr5}
\sum_{q=0}^{n-k-1} (t^{-q})^{\ell_{\mathbf{b}}^{-1} r_g^{-1}} \leq  \sum_{q=0}^{n-k-1} \big(\alpha^{-\ell_{\mathbf{b}}^{-1}}\big)^q \leq \frac{1}{1-\alpha^{-\ell_{\mathbf{b}}^{-1}}} \leq \frac{\alpha \ell_{\mathbf{b}}^{-1}}{\alpha^{\ell_{\mathbf{b}}^{-1}}-1} \ell_{\mathbf{b}}  \leq \frac{\alpha}{\log(\alpha)} \ell_{\mathbf{b}},
\end{align}
where the last step follows from the fact that $f(x) = \frac{x}{\alpha^x-1}$ is a decreasing function and $\lim_{x \downarrow 0} f(x) = \frac{1}{\log \alpha}$. Hence, combining \eqref{eqnr1}, \eqref{eqnr4} and \eqref{eqnr5} gives
\[ \begin{split}
\lambda_n (I_0(\varepsilon)) \le \ & \bigg( 2 + \frac{K_2K_5}{1-1/\alpha} + K_2 K_5 \bigg) \varepsilon^{1/r_{\max}} + K_2 K_5 \sum_{k=1}^{n-1} \sum_{q=0}^{n-k-1} \sum_{\mathbf{b} \in \Sigma_B^k} \sum_{g \in \Sigma_G} p_{\mathbf{b}} p_g \big(\varepsilon t^{-q}\big)^{\ell_{\mathbf{b}}^{-1} r_g^{-1}}\\
\le \ & \bigg( 2 + K_2K_5 \frac{2\alpha-1}{\alpha-1} \bigg) \varepsilon^{1/r_{\max}} + K_2 K_5 \sum_{k=1}^{n-1} \sum_{\mathbf{b} \in \Sigma_B^k} \sum_{g \in \Sigma_G} p_{\mathbf{b}} p_g \varepsilon^{\ell_{\mathbf{b}}^{-1} r_g^{-1}} \frac{\alpha \ell_{\mathbf b}}{\log (\alpha)}\\
\le \ & K_6 \sum_{g \in \Sigma_G} p_g \sum_{k=0}^{n-1} \sum_{\mathbf{b} \in \Sigma_B^k}  p_{\mathbf{b}} \ell_{\mathbf b}\varepsilon^{\ell_{\mathbf{b}}^{-1} r_g^{-1}},
\end{split}\]
where $K_6 = \frac1{\min\{ p_g \, : \, g \in \Sigma_G\}} \big( 2 + K_2K_5 \frac{2\alpha-1}{\alpha-1} \big) + \frac{K_2K_5 \alpha}{\log \alpha}$.
\end{proof}

We are now ready to prove Theorem \ref{thrm3.1}.
\begin{proof}[Proof of Theorem \ref{thrm3.1}] Let $A \subseteq [0,1]$ be a Borel set. First suppose that $\lambda(A) \geq \frac{\varepsilon_0}{3}$. Then there exists a constant $C = C(\frac{\varepsilon_0}{3}) > 0$ such that
\begin{align}
\lambda_n(A) \leq 1 \leq C  \sum_{g \in \Sigma_G} p_g \sum_{k=0}^{\infty} \sum_{\mathbf{b} \in \Sigma_B^k} p_{\mathbf{b}} \ell_{\mathbf{b}} \cdot \lambda(A)^{\ell_{\mathbf{b}}^{-1} r_g^{-1}}.
\end{align}
Now suppose that $\lambda(A) < \frac{\varepsilon_0}{3}$ and set $\varepsilon = 3\lambda(A)$. It follows from Proposition \ref{prop0.6} that for all $n \in \mathbb{N}$ and all $\mathbf u \in \Sigma^n$ we have
\begin{align*}
\lambda\big(T_{\mathbf u}^{-1} A\big) \leq K_1 \big(\lambda(T_{\mathbf u}^{-1}I_0(\varepsilon)) + \lambda(T_{\mathbf u}^{-1}I_c(\varepsilon))\big).
\end{align*}
Together with \eqref{q:measureIm} and Proposition \ref{prop15} this yields for all $n \in \mathbb{N}$ that
\[ \lambda_n(A) \le K_1\cdot (K_3 + K_6) \sum_{g \in \Sigma_G} p_g \sum_{k=0}^{\infty} \sum_{\mathbf{b} \in \Sigma_B^k} p_{\mathbf{b}} \ell_{\mathbf{b}} \cdot \varepsilon^{\ell_{\mathbf{b}}^{-1} r_g^{-1}}.\]
This gives the result.
\end{proof}

\section{Further results and final remarks}\label{sec:cr}

\subsection{Proof of Corollaries \ref{cor2} and \ref{cor}}\label{sec3.4}
In this section we prove Corollaries \ref{cor2} and \ref{cor}.

\begin{proof}[Proof of Corollary~\ref{cor2}]
We use the bound \eqref{eq:2.10} obtained in Theorem \ref{MAIN2}. For convenience, we set $\ell = \ell_{\max}$ and $x=\lambda(A)^{1/r_{\max}}$. The asymptotics is determined by the interplay between $\theta^k\searrow0$ and $x^{1/\ell^k}\nearrow 1$. First suppose $\theta < x^{1/\ell}$. Then $\lambda(A) > \theta^{\ell r_{\max}}$, so there exists a constant $C = C( \theta^{\ell r_{\max}}) > 0$ such that
\[ \mu_{\mathbf p}(A)\le C \cdot \frac1{ \log^{\varkappa} (1/\lambda(A))}.\]
Now suppose $\theta \geq x^{1/\ell}$. Note that $\theta^N\ge x^{1/\ell^N}$ if and only if
\[ \log N + N \log \ell \le \log \bigg( \frac{\log x}{\log \theta} \bigg).\]
Since $\log N \le N$, this last inequality is satisfied if we take for example
\begin{equation}\label{q:estimateN}
N=\left\lfloor \frac 1{1 + \log\ell }\log \bigg( \frac {\log x}{\log\theta} \bigg) \right\rfloor= \left\lfloor \frac 1{1+\log\ell }\log \bigg( \frac {\log(1/x)}{\log (1/\theta)} \bigg)\right\rfloor,
\end{equation}
where $\lfloor y \rfloor$ denotes the largest integer not exceeding $y$. Taking $N$ as in \eqref{q:estimateN}, note that it follows from $\theta \geq x^{1/\ell}$ that $N \geq 0$. Then $\theta^k\ge x^{1/\ell^k}$ for all $k\le N$ as well, and hence
\[\aligned
\sum_{k=0}^{\infty} \theta^k x^{1/\ell^k} &= \sum_{k=0}^{N} \theta^k x^{1/\ell^k} +
\sum_{k=N+1}^{\infty} \theta^k x^{1/\ell^k} \le \sum_{k=0}^{N} \theta^k \cdot x^{1/\ell^N} 
+\sum_{k=N+1}^{\infty} \theta^k \cdot 1\\
&\le \frac {1}{1-\theta} x^{1/\ell^N} +\frac {\theta^{N+1} }{1-\theta}\le \frac {1}{1-\theta} (1+\theta) \theta^N.
\endaligned
\]
From \eqref{q:estimateN} we see that $N \ge \frac 1{1+\log\ell }\log \big( \frac {\log x}{\log\theta} \big) -1$, thus
\[
\begin{split}
\theta^N =\ & \exp(N\log\theta) \le \exp\bigg(\bigg(\frac 1{1+\log\ell }\log \bigg(\frac {\log(1/x)}{\log (1/\theta)}\bigg) -1 \bigg)\log\theta\bigg) \\
=\ & \exp\bigg(  \frac {\log\theta}{1+\log\ell }\log {\log(1/x)} + C(\ell,\theta)\bigg)\\
=\ & \overline{C}(\ell,\theta) \big( \log( 1/x)\big)^{\frac {\log\theta}{1+\log \ell}} = \overline{C}(\ell,\theta)  \bigg( \frac{r_{\max}} {\log(1/\lambda(A))} \bigg)^\varkappa,
\end{split}
\]
where we set $\varkappa =\frac {\log(1/\theta)}{1+\log\ell}>0$, and where $C(\ell,\theta) \in \mathbb{R}$ and $\overline{C}(\ell,\theta) > 0$ are constants that only depend on $\ell$ and $\theta$.
We conclude from the bound \eqref{eq:2.10} that 
\[ \mu_{\mathbf p}(A)\le K \cdot \frac1{ \log^{\varkappa} (1/\lambda(A))}\]
for some positive constant $K$.
\end{proof}

The proof of Corollary \ref{cor} consists of two steps. Firstly we show that any weak limit point of $\mu_{\mathbf p_n}$ is a stationary measure, i.e., satisfies \eqref{eqn9}, and secondly that any weak limit point of $\mu_{\mathbf p_n}$ is absolutely continuous with respect to the Lebesgue measure. The corollary then follows from the uniqueness of absolutely continuous stationary measures given by Theorem \ref{MAIN}.

\begin{proof}[Proof of Corollary~\ref{cor}]
For each $n \ge 0$, let $\mathbf p_n = (p_{n,j})_{j \in \Sigma}$ be a positive probability vector such that $\sup_{n}\sum_{b \in \Sigma_B} p_{n,b}  \ell_b<1$ and assume that $\lim_{n\to \infty}\mathbf p_n=\mathbf p$ in $\mathbb R_+^N$ for some $\mathbf p= (p_j)_{j \in \Sigma}$. Let $\tilde\mu$ be a weak limit point of $\mu_{\mathbf p_n}$. Again, note that such a $\tilde{\mu}$ exists because the space of probability measures on $[0,1]$ equipped with the weak topology is sequentially compact. After passing to a subsequence we have for any continuous function $\varphi:[0,1]\to \mathbb R$ that
\[ \lim_{n \to \infty}\int_{[0,1]} \varphi \, d\mu_{\mathbf p_n} = \int_{[0,1]} \varphi \, d\tilde\mu.\]
Moreover, by the stationarity of the measures $\mu_{\mathbf p_n}$ it follows that for each $n \ge 1$,
\[ \int_{[0,1]} \varphi \, d\mu_{\mathbf p_n} =\sum_{j \in \Sigma} p_{n,j} \int_{[0,1]} \varphi \circ T_j \, d\mu_{\mathbf p_n}.\]
To prove that $\tilde \mu$ is stationary for $\mathbf p$, it is sufficient to show that for each $j \in \Sigma$,
\begin{equation}\label{eq:intphii}
\lim_{n \to \infty} p_{n,j} \int_{[0,1]}\varphi\circ T_j \, d\mu_{\mathbf p_n} = p_j \int_{[0,1]} \varphi\circ T_j \, d\tilde\mu.
\end{equation}

If $j\in\Sigma_B$ this is obvious, since then $\varphi\circ T_j$ is continuous. For $j \in \Sigma_G$ the map $\varphi\circ T_j$ might have a discontinuity at $c$. In this case, we let $\varphi_\delta$ be the continuous function given by $\varphi_\delta (x) = \varphi\circ T_j(x)$ for $x\in I\setminus(c-\delta, c+\delta)$ and $\varphi_\delta$ is linear otherwise. Then we have 
\[
\lim_{n \to \infty} \left| p_{n,j}\int_{[0,1]} \varphi_\delta \, d\mu_{\mathbf p_n}-p_{j}\int_{[0,1]}\varphi_\delta \, d\tilde{\mu}\right|=0,
\]
by the weak convergence and since $p_{n,j}\to p_{j}$ as $n \rightarrow \infty$. Also, we have
\[
\left| p_{n,j} \int_{[0,1]}\varphi\circ T_j \, d\mu_{\mathbf p_n}- p_{n,j} \int_{[0,1]}\varphi_\delta \,d\mu_{\mathbf p_n}\right| \le C\mu_{\mathbf p_n}([c-\delta,c+\delta]) \to 0 \text{ as } \delta\to 0,
\]
where the convergence is uniform in $n$ because of \eqref{eq:2.10}. Similarly,
\[
\left| p_j \int_{[0,1]}\varphi\circ T_j d\tilde{\mu}- p_j \int_{[0,1]}\varphi_\delta d\tilde{\mu}\right| \le C\tilde{\mu}([c-\delta,c+\delta]) \to 0 \text{ as } \delta\to 0,
\]
 The last three relations imply \eqref{eq:intphii}.

\medskip
To show that $\tilde \mu$ is absolutely continuous with respect to the Lebesgue measure $\lambda$ we proceed as in the proof of Theorem~\ref{MAIN2}.  We set $\tilde{\theta} = \sup_{n}\sum_{b \in \Sigma_B} p_{n, b}  \ell_b<1$. Let $A \subseteq [0,1]$ be a Borel set. By Theorem \ref{MAIN} every $\mu_{\mathbf p_n}$ satisfies \eqref{eq:2.10}, so that
\[\mu_{\mathbf p_n} (A) \le C_n \sum_{k=0}^{\infty} \tilde \theta^k \lambda(A)^{\ell_{\max}^{-k}r_{\max}^{-1}},\]
where the constant $C_n$ depends on $(\sum_{g \in \Sigma_G} p_{n,g}\big)^{-1}$ and $(\min \{ p_{n,g} \, : \, g \in \Sigma_G\})^{-1}$ (and properties of the good and bad maps themselves that are not linked to the probabilities). Since each $\mathbf p_n$, $n \ge 0$, is a positive probability vector and $\lim_{n \rightarrow \infty} \mathbf p_n = \mathbf p$, both these quantities can be bounded from above and $\tilde C:= \sup_n C_n < \infty$. From the weak convergence of $\mu_{\mathbf p_n}$ to $\tilde \mu$ we obtain as in \eqref{eqn3.42} using the Portmanteau Theorem that
\[ \tilde \mu(A) \le \tilde C \sum_{k=0}^{\infty} \tilde \theta^k \lambda(A)^{\ell_{\max}^{-k}r_{\max}^{-1}}.\]
Hence, $\tilde \mu \ll \lambda$. By Theorem~\ref{MAIN} we know that $\mu_{\mathbf p}$ is the unique acs probability measure for $F$ and $\mathbf p$. So, $\tilde\mu=\mu_{\mathbf p}$.
\end{proof}

\subsection{The non-superattracting case}\label{ss:nonsa}
With some modifications the results from Theorem \ref{MAIN} and Theorem~\ref{MAIN2} can be extended to the class $\mathfrak{B}^1 \supseteq \mathfrak B$ of bad maps of which critical order $\ell_b$ in (B3) is allowed to be equal to 1. We will list the modified statements and the necessary modifications to the proofs here. Note that for each $T \in \mathfrak{B}^1 \setminus \mathfrak B$, we have $DT(c) \neq 0$, and due to the minimal principle, $|DT(c)| < 1$. So we consider $T_1, \ldots, T_N \in \mathfrak G \cup \mathfrak B^1$ with $\Sigma_B^1 = \{1 \leq j \leq N: T_j \in \mathfrak B^1\}$ and $\Sigma_G$, $\Sigma_B$ as before and such that $\Sigma_G, \Sigma_B^1 \backslash \Sigma_B \neq \emptyset$. Furthermore, we write again $\Sigma = \{1,\ldots,N\} = \Sigma_G \cup \Sigma_B^1$.

\begin{theorem}\label{MAIN3} Let $\{T_j: j \in \Sigma\}$ be as above and $\mathbf p = (p_j)_{j \in \Sigma}$ a positive probability vector.
 \begin{enumerate}
\item There exists a unique (up to scalar multiplication) stationary $\sigma$-finite measure $\mu_{\mathbf p}$ for $F$ that is absolutely continuous with respect to the one-dimensional Lebesgue measure $\lambda$. This measure is ergodic and the density $\frac{d\mu_{\mathbf p}}{d\lambda}$ is bounded away from zero and is locally Lipschitz on $(0,c)$ and $(c,1)$.
\item Suppose $\ell_{\max} > 1$. 
\begin{enumerate}[(i)]
\item The measure $\mu_{\mathbf p}$ is finite if and only if $\theta = \sum_{b \in \Sigma_B^1} p_b \ell_b <1$. In this case, for each $\hat{\theta} \in (\theta,1)$ there exists a constant $C(\hat{\theta}) > 0$ such that
\begin{align}\label{eq:5.2}
\mu_{\mathbf p}(A) \leq C(\hat{\theta}) \cdot \sum_{k=0}^{\infty} \hat{\theta}^k \lambda(A)^{\ell_{\max}^{-k} r_{\max}^{-1}}
\end{align}
for any Borel set $A \subseteq [0,1]$, where $r_{\max} = \max\{r_g: g \in \Sigma_G\}$ and $\ell_{\max} = \max\{\ell_b: b \in \Sigma_B\}$.
\item The density $\frac{d\mu_{\mathbf p}}{d\lambda}$ is not in $L^q$ for any $q > 1$. 
\end{enumerate}
\item Suppose $\ell_{\max} = 1$. 
\begin{enumerate}[(i)]
\item The measure $\mu_{\mathbf p}$ is finite, and for each $\bm \eta = (\eta_b)_{b \in \Sigma_B^1}$ such that $\eta_b > 1$ for each $b \in \Sigma_B^1$ and $\hat{\theta}(\bm \eta) = \sum_{b \in \Sigma_B^1 } p_b \eta_b < 1$ there exists a constant $C(\bm \eta) > 0$ such that
\begin{align}\label{eq:5.3}
\mu_{\mathbf p}(A) \leq C(\bm \eta) \cdot \sum_{k=0}^{\infty} \hat \theta(\bm \eta)^k \lambda(A)^{\eta_{\max}^{-k} r_{\max}^{-1}}
\end{align}
for any Borel set $A \subseteq [0,1]$, where $\eta_{\max} = \max\{\eta_b: b \in \Sigma_B^1\}$. If $\sum_{b \in \Sigma_B^1} \frac{p_b}{|DT_b(c)|} < 1$, so if the bad maps are expanding on average at the point $c$, then we can get the estimate
\begin{align}\label{eq:5.4}
\mu_{\mathbf p}(A) \leq C \cdot \lambda(A)^{r_{\max}^{-1}}
\end{align}
for some constant $C > 0$ and any Borel set $A \subseteq [0,1]$.
\item If $r_{\max} > 1$, then $\frac{d\mu_{\mathbf p}}{d\lambda} \not \in L^q$ for any $q \geq \frac{r_{\max}}{r_{\max}-1}$. If, moreover, $\sum_{b \in \Sigma_B^1} \frac{p_b}{|DT_b(c)|} < 1$, then $\frac{d\mu_{\mathbf p}}{d\lambda} \in L^q$ for all $1 \leq q < \frac{r_{\max}}{r_{\max}-1}$.
\item If $r_{\max} =1$ and $\sum_{b \in \Sigma_B^1} \frac{p_b}{|DT_b(c)|} < 1$, then $\frac{d\mu_{\mathbf p}}{d\lambda} \in L^\infty$.
\end{enumerate}
\end{enumerate}
\end{theorem}

\medskip

The main issue we need to deal with in order to get Theorem~\ref{MAIN3} is adapting Lemma~\ref{lemma3.6}, i.e., finding suitable bounds for $|T_{\omega}^n(x)-c|$, since the constants $\tilde K$ and $\tilde M$ from Lemma~\ref{lemma3.6} are not well defined in case $\ell_{\min} = 1$. This is done in the next two lemmata. For the upper bound of $|T_{\omega}^n(x)-c|$ we assume $\ell_{\max} > 1$ since we only need it for the proof of part (2)(i).

\begin{lemma}\label{lemma5.1}
Let $\{T_j: j \in \Sigma\}$ be as above. Suppose $\ell_{\max} > 1$. There are constants $\hat{M} > 1$ and $\delta > 0$ such that for all $n \in \mathbb{N}$, $\omega \in (\Sigma_B^1)^{\mathbb{N}}$ and $x \in [c-\delta,c+\delta]$ we have
\begin{align}\label{eq5.2}
|T_{\omega}^n(x)-c| \leq \Big(\hat{M} |x-c|\Big)^{\ell_{\omega_1} \cdots \ell_{\omega_n}}.
\end{align}
\end{lemma}

\begin{proof}
Similar as in the proof of Lemma \ref{lemma3.6} it follows that there exists an $M >1$ such that for any $b \in \Sigma_B$ and $x \in [0,1]$ we have
\begin{align}\label{eq5.7}
|T_b(x)-c| \leq M |x-c|^{\ell_b}.
\end{align}
Furthermore, there exists a $\delta > 0$ such that $|DT_b(x)| < 1$ for all $x \in [c-\delta,c+\delta]$ and $b \in \Sigma_B^1$. This implies
\begin{align}\label{eq5.8}
|T_b(x)-c| < |x-c|
\end{align}
 for all $x \in [c-\delta,c+\delta]$ and $b \in \Sigma_B^1$. Note that $\Sigma_B \neq \emptyset$ because $\ell_{\max} > 1$. We set $\upsilon = \min\{\ell_b: b \in \Sigma_B\} > 1$ and $\hat{M} = M^{\frac{1}{\upsilon-1}}$. For each $n \in \mathbb{N}$ and $\omega \in (\Sigma_B^1)^{\mathbb{N}}$, write
\begin{align}
m(n,\omega) = \#\{1 \le \omega_i\le n \, :\, \ell_{\omega_i} > 1\}.
\end{align}
The statement follows by showing that for all $n \in \mathbb{N}$, $\omega \in (\Sigma_B^1)^{\mathbb{N}}$ and $x \in [c-\delta,c+\delta]$ we have
\begin{align}\label{eqn5.9}
|T_{\omega}^n(x)-c| \leq \Big(M^{(1-\upsilon^{-m(n,\omega)})/(\upsilon-1)}|x-c|\Big)^{\ell_{\omega_1} \cdots \ell_{\omega_n}}.
\end{align}
We prove \eqref{eqn5.9} by induction. From \eqref{eq5.7} and \eqref{eq5.8} it follows that \eqref{eqn5.9} holds for $n=1$. Now suppose \eqref{eqn5.9} holds for some $n \in \mathbb{N}$. Let $\omega \in (\Sigma_B^1)^{\mathbb{N}}$ and $y \in [c-\delta,c+\delta]$. If $\ell_{\omega_{n+1}} = 1$, then the desired result follows by applying \eqref{eq5.8} with $j = \omega_{n+1}$ and $x = T_{\omega}^n(y)$. Suppose $\ell_{\omega_{n+1}} > 1$. Then, using \eqref{eq5.7},
\begin{align*}
|T_{\omega}^{n+1}(y) -c| & \leq M |T_{\omega}^n(y)-c|^{\ell_{\omega_{n+1}}} \\
& \leq \Big(M^{(1-\upsilon^{-m(n,\omega)})/(\upsilon-1)+\upsilon^{-m(n+1,\omega)}}|y-c|\Big)^{\ell_{\omega_1} \cdots \ell_{\omega_{n+1}}}.
\end{align*}
Using that
\begin{align}
\upsilon^{-m(n+1,\omega)} = \frac{\upsilon^{-m(n,\omega)}-\upsilon^{-m(n+1,\omega)}}{\upsilon-1},
\end{align}
the desired result follows.
\end{proof}

\begin{lemma}\label{lemma5.2}
Let $\{T_j: j \in \Sigma\}$ be as above. Let $\bm{\eta} = (\eta_b)_{b \in \Sigma_B^1}$ be a vector such that $\eta_b > 1$ for each $b \in \Sigma_B^1$. Set $\hat{\eta}_b = \max\{\eta_b,\ell_b\}$ for each $b \in \Sigma_B^1$. Then there exists a constant $\hat{K}(\bm{\eta}) \in (0,1)$ such that for all $n \in \mathbb{N}$, $\omega \in (\Sigma_B^1)^{\mathbb{N}}$ and $x \in [0,1]$ we have
\begin{align}
\Big(\hat{K}(\bm{\eta})|x-c|\Big)^{\hat{\eta}_{\omega_1} \cdots \hat{\eta}_{\omega_n}} \leq |T_{\omega}^n(x)-c|.
\end{align}
\end{lemma}

\begin{proof}
Note from (B3) that for each $b \in \Sigma_B^1$ we have
\begin{align*}
K_b |x-c|^{\hat{\eta}_b-1} \leq K_b |x-c|^{\ell_b-1} \leq |DT_b(x)|.
\end{align*}
The result now follows in the same way as in the proof of Lemma \ref{lemma3.6} by setting $\hat{\eta}_{\min} = \min\{\hat{\eta}_b : b \in \Sigma_B^1\}$, $\hat{\eta}_{\max} = \max\{\hat{\eta}_b : b \in \Sigma_B^1\}$ and $\hat{K}(\bm{\eta}) = \big( \frac{\min \{ K_b \, : \, b \in \Sigma_B^1 \}}{\hat{\eta}_{\max}}\big)^ \frac1{\hat{\eta}_{\min}-1}$.
\end{proof}

\begin{proof}[Proof of Theorem \ref{MAIN3}]
Firstly, note that (1), (2)(ii) and the first part of (3)(ii) immediately follow from Remark \ref{remark3.1}. Moreover, as in \cite[Section 5.4]{dMvS93} it can be shown that \eqref{eq:5.4} implies that $\frac{d\mu_{\mathbf p}}{d\lambda}$ is in $L^q$ if $r_{\max} > 1$ and $1 \leq q < \frac{r_{\max}}{r_{\max}-1}$, giving the remainder of 3(ii). It is immediate that \eqref{eq:5.4} implies that $\frac{d\mu_{\mathbf p}}{d\lambda}$ is in $L^{\infty}$ if $r_{\max} = 1$, so (3)(iii) holds. Hence, it remains to prove (2i) and (3i).

\medskip
Suppose $\theta = \sum_{b \in \Sigma_B^1} p_b \ell_b \geq 1$, which means that $\ell_{\max} > 1$. The proof that in this case $\mu_{\mathbf{p}}$ is infinite follows by the same reasoning as in Subsection \ref{subsec4.1} by now taking $\gamma = \min\{\delta,\frac{1}{2}\hat{M}^{-1}\}$ with $\delta$ and $\hat{M}$ as in the proof of Lemma \ref{lemma5.1}. Now suppose $\theta < 1$. Let $\bm{\eta} = (\eta_b)_{b \in \Sigma_B^1}$ be a vector such that $\eta_b > 1$ for each $b \in \Sigma_B^1$ and $\hat{\theta}(\bm{\eta}) = \sum_{b \in \Sigma_B^1} p_b \hat{\eta}_b < 1$ with again $\hat{\eta}_b = \max\{\eta_b,\ell_b\}$. Applying Lemma \ref{lemma5.2} yields that for all $\varepsilon >0$, $n \in \mathbb N$, $\mathbf b \in (\Sigma_B^1)^n$,
\begin{equation}\label{eq5.12}
T_{\mathbf b}^{-1} \big( I_c(\varepsilon) \big) \subseteq I_c \big( \hat{K}(\bm{\eta})^{-1}\varepsilon^{\hat{\eta}_{\mathbf b}^{-1}} \big),
\end{equation}
where we used the notation $\hat{\eta}_{\mathbf b} = \hat{\eta}_{b_1} \cdots \hat{\eta}_{b_n}$ for a word $\mathbf b = b_1 \cdots b_n$. Following the line of reasoning in Subsection \ref{sec3.3} with \eqref{eq5.12} instead of \eqref{eq3.57}, we obtain that there exists a constant $C(\bm{\eta}) > 0$ such that
\begin{align}\label{eq:5.13}
\mu_{\mathbf p}(A) \leq C(\bm \eta) \cdot \sum_{k=0}^{\infty} \hat\theta(\bm \eta)^k \lambda(A)^{\hat \eta_{\max}^{-k} r_{\max}^{-1}}
\end{align}
for any Borel set $A \subseteq [0,1]$. In case $\ell_{\max} > 1$ we can choose $\bm \eta$ to satisfy $\hat{\eta}_{\max} = \ell_{\max}$ and such that $\hat{\theta}(\bm \eta) -\theta$ is arbitrarily small, which yields (2)(i). In case $\ell_{\max} = 1$, then $\hat{\eta}_{\max} = \eta_{\max}$, so this together with \eqref{eq:5.13} yields the first part of (3)(i).

\medskip
Finally, for the second part of (3)(i), suppose $\ell_{\max} = 1$ and $\Lambda = \sum_{b \in \Sigma_B^1} \frac{p_b}{|DT_b(c)|} < 1$. Setting $K_{\mathbf{b}} = |DT_{\mathbf{b}}(c)|$ for each $\mathbf{b} \in (\Sigma_B^1)^n$ and $n \in \mathbb{N}$, note that for all $\varepsilon > 0$, $n \in \mathbb{N}$, $\mathbf{b} \in (\Sigma_B^1)^n$,
\begin{equation}\label{eq5.14}
T_{\mathbf b}^{-1} \big( I_c(\varepsilon) \big) \subseteq I_c \big( K_{\mathbf{b}}^{-1} \varepsilon \big).
\end{equation}
By using \eqref{eq5.14} instead of \eqref{eq3.57}, letting $\tilde{p}_{\mathbf b} = K_{\mathbf b}^{-1} p_{\mathbf b}$ play the role of $p_{\mathbf b}$ in the reasoning of Subsection \ref{sec3.3} and noting that $\Lambda^k = \sum_{\mathbf b \in (\Sigma_B^1)^k} \tilde{p}_{\mathbf b}$, we arrive similarly as for Theorem \ref{thrm3.1} to the conclusion that there exists a constant $\tilde{C} > 0$ such that for all $n \in \mathbb{N}$ and all Borel sets $A \subseteq [0,1]$,
\begin{align}\label{eqn5.15}
\lambda_n(A) \leq \tilde{C} \cdot \sum_{g \in \Sigma_G} p_g \Big(\sum_{k=0}^{\infty} \Lambda^k\Big) \lambda(A)^{r_g^{-1}}.
\end{align}
This proves the remaining part of (3)(i).
\end{proof}

\subsection{Final remarks}
The results from Theorem~\ref{MAIN3} contain one possible extension of our main results to another set of conditions (G1)--(G4), (B1)--(B4). In this section we discuss some of the questions that our main results brought up in this respect, i.e., about whether or not some of the conditions (G1)--(G4), (B1)--(B4) can be relaxed, and questions about other possible future extensions.

\medskip
A condition that plays a fundamental role in the proofs of Theorem~\ref{MAIN} and Theorem~\ref{MAIN2} is the fact that the critical point is mapped to a point that is a common repelling fixed point for all maps $T_j$. We considered whether this condition can be relaxed, for instance by assuming that the branches of one of the good maps are not full. However, in this case the critical values of the random system are not just $0,c,1$ but contain all the values of all possible postcritical orbits of $c$. This has several consequences:
\vskip .1cm
- An invariant density (if it exists) clearly cannot be locally Lipschitz on $(0,c)$ and $(c,1)$.
\vskip .1cm
- Proposition \ref{prop0.6} and all subsequent arguments fail, since it is not sufficient to restrict to neighbourhoods around only $0$, $c$ and $1$. One might try to solve this issue by requiring that the postcritical orbits `gain enough expansion' as was done in for instance \cite{NowvSt} for deterministic maps. An analogous condition for random systems, however, would become much stronger since it would have to hold for all possible random orbits of $c$.
\vskip .1cm
- The argument using Kac's Lemma might fail, because in that case there exist words $\mathbf u$ with symbols in $\Sigma$ and neighbourhoods $U$ of $c$ such that $T_{\mathbf{u}}(x)$ is bounded away from zero and one uniformly in $x \in U$.

\medskip
The dynamical behaviour of the system is governed by the interplay between the superexponential convergence at $c$ and the exponential divergence from $0$ and $1$. In this article we fixed the exponential divergence away from $0$ and $1$ and the two regimes $\theta < 1$ and $\theta \geq 1$ in Theorem \ref{MAIN2} only refer to the convergence at $c$: For smaller $\theta$ orbits are less attracted to $c$. It would be interesting to see under what other conditions on the rates of convergence to $c$ and divergence from $0$ and $1$ the system admits an acs measure. Could one for example
\vskip .1cm
- take exponential convergence to $c$ and polynomial divergence from 0 and 1, or 
\vskip .1cm
- replace the conditions (G4) and (B4) stating that all good and bad maps are expanding at $0$ and $1$ by the condition that the random system is expanding on average at a sufficiently large neighbourhood of $0$ and $1$?

\medskip
There are also some additional questions that our main results raise. It would be interesting for example to study further statistical properties of the random system such as mixing properties and if possible mixing rates in 	case the acs measure is finite. It is not clear a priori if the behaviour of the good maps dominates the statistical properties of the random system, since trajectories spend long periods of time near the points $0$, $c$ and $1$. In this respect the dynamics resembles that of the Manneville-Pomeau maps, and mixing rates might be polynomial rather than exponential. A way to approach this problem is by estimating the measures $\mathbb P\times \lambda(\{\varphi_Y>n\})$, where $\varphi_Y$ is the first return time to $Y$ defined in Section \ref{sec:PT2.1} as they give information on the rates of decay of correlations. To obtain the desired decay rates it is sufficient to obtain estimates for $\mathbb P\times \lambda(\{\varphi_Y=k\})$ for all $k> n$. Recall that every returning set $\{\varphi_Y=k\}$ is of the form $C_k\times J(C_k)$, where $C_k\subset \Sigma^{\mathbb N}$ is a cylinder set and $J(C_k) \subset I$ is an interval with return time $k$, which depends only on $C_k$. Obtaining effective estimates on individual intervals $J$ by directly looking at pre-images of $Y$ under the skew product system does not seem very feasible at the moment, since cylinders can contain a positive proportion of bad maps. An alternative approach could be a combinatorial construction as in \cite{ALP} or \cite{BLS}, where a two step induction process is introduced. To perform a similar construction we have to find a suitable way to define the binding period or the slow recurrence to the critical set, which takes into account the existence of bad maps.

\medskip
Finally, in Theorems \ref{MAIN} and \ref{MAIN3} we have seen that the regularity of the density $\frac{d\mu_{\mathbf p}}{d\lambda}$ depends on whether or not there is a bad map for which $c$ is superattracting: If $\ell_{\max} >1$, then $\frac{d\mu_{\mathbf p}}{d\lambda}$ is not in $L^q$ for any $q > 1$. On the other hand, if $\ell_{\max} = 1$ and the bad maps are expanding on average at $c$, i.e.~$\sum_{b \in \Sigma_B^1} \frac{p_b}{|DT_b(c)|} < 1$, then the density has the same regularity as in the setting of Theorem \ref{thrm1.1} by Nowicki and van Strien. Indeed, in this case, if $r_{\max} > 1$, we have $\frac{d\mu_{\mathbf p}}{d\lambda} \in L^q$ if and only if $1 \leq q < \frac{r_{\max}}{r_{\max} -1}$ and in the case that $r_{\max} = 1$ we have $\frac{d\mu_{\mathbf p}}{d\lambda} \in L^q$ for all $q \in [1,\infty]$. In view of this, one could wonder for which $q > 1$ we have $\frac{d\mu_{\mathbf p}}{d\lambda} \in L^q$ in the intermediate case that $\ell_{\max} = 1$ and $\sum_{b \in \Sigma_B^1} \frac{p_b}{|DT_b(c)|} \geq 1$, i.e.~if $c$ is not superattracting for any bad map and the bad maps are not expanding on average at $c$.

\bibliographystyle{plain} 
\bibliography{Bibliography}

\end{document}